\documentclass[11pt, a4paper]{amsart}
\usepackage[english]{babel}
\usepackage{amssymb,amsmath,amsthm,mathrsfs,bm}
\usepackage{color}
\usepackage[margin=1in]{geometry}
\usepackage{stmaryrd}
\usepackage{graphicx}

\newtheorem{theorem}[subsection]{Theorem}
\newtheorem{example}[subsection]{Example}
\newtheorem{conjecture}[subsection]{Conjecture}
\newtheorem{definition}[subsection]{Definition}
\newtheorem{lemma}[subsection]{Lemma}
\newtheorem{remark}[subsection]{Remark}
\newtheorem{proposition}[subsection]{Proposition}
\newtheorem{corollary}[subsection]{Corollary}

\newtheorem*{claim*}{Claim}
\newtheorem*{theorem*}{Theorem}

\def\bal{\begin{aligned}}
\def\eal{\end{aligned}}
\def\be{\begin{equation}\label}
\def\ee{\end{equation}}
\def\bcs{\begin{cases}}
\def\ecs{\end{cases}}
\def\={\;=\;}
\def\+{\,+\,}
\def\-{\,-\,}

\def\Z{{\mathbb Z}}

\def\Q{{\mathbb Q}}
\def\R{{\mathbb R}}

\def\lb{\llbracket}
\def\rb{\rrbracket}
\def\ord{\mathrm{ord}}

\def\sF{\mathcal{F}}

\def\cartier{\mathscr{C}_p}
\def\fil{\mathscr{F}}

\def\v#1{{\bf #1}}

\def\is{\equiv}
\def\mod#1{({\rm mod}\ #1)}
\def\ceil#1{\lceil #1\rceil}

\def\boldomega{\bm{\omega}}
\def\pow#1{\llbracket #1\rrbracket}
\def\hat{\widehat}
\def\An{A\textsubscript{n}}

\title{Dwork crystals III: from excellent Frobenius lifts towards supercongruences}
\author{Frits Beukers, Masha Vlasenko}

\address{Utrecht University }
\email{f.beukers@uu.nl}

\address{Institute of Mathematics of the Polish Academy of Sciences}
\email{m.vlasenko@impan.pl}

\thanks{
Work of Frits Beukers was supported by the Netherlands Organisation
for Scientific Research (NWO), grant TOP1EW.15.313. Work of Masha Vlasenko was supported by the National Science Centre of Poland (NCN), grant UMO-2020/39/B/ST1/00940.}

\begin{document}
\maketitle

\begin{abstract}
This paper is a continuation of our Dwork crystals series. Here we exploit the Cartier operation to
prove supercongruences for expansion coefficients of rational functions.
In the process it appears that excellent Frobenius lifts are a driving force behind supercongruences.
Originally introduced by Dwork, these excellent lifts have occurred rather infrequently in the literature,
and only in the context of families of elliptic curves and abelian varieties. In the final sections of
this paper we present a list of examples that occur in the case of families of Calabi-Yau varieties.


\end{abstract}

\section{Introduction}
We recall the main points of our earlier papers \cite{BeVl20I} and \cite{BeVl20II}
titled Dwork crystals I and II.
In this paper we shall refer to them as Part I and Part II.
Let $p$ be an odd prime and $R$ \emph{be a $p$-adic ring}.
By that we mean a characteristic zero domain such that $\cap_{s\ge1}p^{s}R=\{0\}$ 
and which is $p$-adically complete. Common examples are $\Z_p,\Z_p\lb t\rb$ and
the $p$-adic completion of $\Z[t,1/P(t)]$ where $P(t) \in \Z[t]$ is a polynomial not divisible by $p$.
We also assume our ring is equipped with a so-called \emph{Frobenius lift $\sigma$},
i.e. an endomorphism
$\sigma:R\to R$ such that $\sigma(r)\is r^p\mod{p}$ for all $r\in R$. 

The main player is a Laurent polynomial $f(\v x)= \sum_{\v u} f_\v u \v x^\v u$ in the variables $\v x=(x_1,\ldots,x_n)$ 
with coefficients $f_{\v u} \in R$. The Newton polytope of $f$ is the convex hull of the
support $Supp(f) = \{ \v u \in \Z^n \;|\; f_\v u \ne 0\}$ and is denoted by $\Delta \subset \R^n$.  On $\Delta$ we have 
a rudimentary topology
whose closed sets are the faces of $\Delta$. An example of an open set is given by the interior $\Delta^\circ = \Delta \setminus \{ \text{ all proper faces } \}$.

Let $\mu\subset\Delta$ be any open subset. 
We consider the following $R$-module consisting of rational functions of a special shape: 
\[
\Omega_f(\mu) = \left\{ (m-1)!\frac{A(\v x)}{f(\v x)^{m}}\quad \Big| \quad m \ge 1 \text{ and } \mbox{ Supp}(A)\subset m\mu. \right\}
\]
When $\mu=\Delta$ we simply write $\Omega_f(\Delta)=\Omega_f$. By
$\hat\Omega_f(\mu)$ we denote the $p$-adic completion of $\Omega_f(\mu)$. In Part I
we have defined the $R$-linear \emph{Cartier operator} 
\[
\cartier:\hat\Omega_f(\mu)\to\hat\Omega_{f^\sigma}(\mu).
\]
Here $f^\sigma = \sum_\v u f_\v u^\sigma \v x^\v u$ denotes $f$ with $\sigma$ applied to its coefficients.
One can expand the elements of $\hat\Omega_f$ as
formal Laurent series, and our Cartier operation has a particularly simple expression on
such expansions:
\[
\cartier\left(\sum_{\v u}a_{\v u}\v x^{\v u}\right)=\sum_{\v u}a_{p\v u}\v x^{\v u}.
\]
This is explained in \S2 of Part I. Using $\cartier$ we defined the submodule
\[
\fil_1(\mu)=\{\omega\in\hat\Omega_f(\mu)\quad|\quad\cartier^s(\omega)\in p^s\hat
\Omega_{f^{\sigma^s}}(\mu),\ s\ge1\}.
\] 
From this definition one can immediately see 
that $\cartier(\fil_1(\mu))\subset p\fil_1^\sigma(\mu)$, where $\fil_1^\sigma(\mu)$ denotes a similar submodule in $\hat\Omega_{f^\sigma}(\mu)$. The main result of Part I is the following.

\begin{theorem}[\cite{BeVl20I}, Theorem 4.3 and Remark 4.6]\label{main-Part-I} If the Hasse--Witt matrix relative to $\mu$ is invertible over $R$, then we have the direct sum decomposition of $R$-modules
\[
\hat\Omega_f(\mu)\cong \Omega_f^{(1)}(\mu)\oplus\fil_1(\mu),
\]
where $\Omega_f^{(1)}(\mu)$ is the free $R$-module generated by the elements $\frac{\v x^{\v u}}{f(\v x)}$
with $\v u\in\mu \cap \Z^n$. Moreover, 
\[
\cartier\left(\hat\Omega_f(\mu)\right)\subset\Omega_{f^\sigma}^{(1)}(\mu)+p\fil_1^\sigma(\mu).
\]
\end{theorem}
 
The Hasse--Witt matrix in this theorem is given by 
\[
HW_{\v u, \v v} = \text{the coefficient of } \v x^{p \v v - \v u} \text{ in } f(\v x)^{p-1}, \quad \v u, \v v \in \mu \cap \Z^n.
\]   
Restricted modulo $p$, the Cartier operation yields a map from $\Omega_{f}^{(1)}(\mu)$ to $\Omega_{f^\sigma}^{(1)}(\mu)$ which is represented in the respective bases by the matrix $HW=HW(\mu)$. The condition of invertibility of the Hasse--Witt matrix is thus equivalent to the Cartier operation being an isomorphism modulo $p$. 

The first main result of this paper is Theorem~\ref{free-quotient-k} which gives a generalization of 
Theorem~\ref{main-Part-I}. It states that for $1 \le k < p$ under certain condition, which we 
call the $k$th Hasse--Witt condition, one has a decomposition
\[
\hat\Omega_f(\mu)\cong \Omega_f^{(k)}(\mu)\oplus\fil_k(\mu),
\]
where $\Omega_f^{(k)}(\mu)$ is the free $R$-module generated by the elements $\frac{\v x^{\v u}}{f(\v x)^k}$ with $\v u \in (k \mu) \cap \Z^n$ and  
\[
\fil_k(\mu)=\{\omega\in\hat\Omega_f(\mu)\quad|\quad\cartier^s(\omega)\in p^{ks}\hat
\Omega_{f^{\sigma^s}}(\mu),\ s\ge1\}.
\]
Similarly to Theorem~\ref{main-Part-I}, we will also show that 
\[
\cartier\left(\hat\Omega_f(\mu)\right)\subset\Omega_{f^\sigma}^{(k)}(\mu)+p^k\fil_k^\sigma(\mu).
\]
The $k$th Hasse-Witt condition means that the image of the Cartier operation modulo $p^k$ is as maximal as possible.
In Section~\ref{sec:HW-matrices} it will be interpreted as invertibility of certain determinants, see Proposition~\ref{hwc-maximal} and Corollary~\ref{main-theorem-alt}. We will prove these results for submodules of $\Omega_f$ which are slightly more general that $\Omega_f(\mu)$. They will be introduced in Section~\ref{sec:dwork-crystals}. 

Let us now explain our motivation and some applications of the above results. Since $\cartier$ maps $\fil_k$
into $\fil_k^\sigma$, it induces an $R$-linear map 
\[
\cartier: \hat\Omega_f(\mu)/\fil_k \to \hat\Omega_{f^\sigma}(\mu)
/\fil_k^\sigma.
\]
When the $k$th Hasse-Witt condition holds, this is a map between
free modules of rank $\#((k \mu) \cap \Z^n)$. Note that we use the somewhat lazy, but suggestive, notation
$\hat\Omega_f(\mu)/\fil_1$ for what should be $\hat\Omega_f(\mu)/\fil_1(\mu)$.
The determination of the matrix of $\cartier$ between these finite rank modules
is a principal goal of our considerations.
In Parts I, II we have described this matrix for $k=1$ as a $p$-adic limit of quotients of Hasse-Witt type matrices, which leads to Dwork-type congruences. A recurring example of such congruences comes from 
$f=1-tg(\v x)$, where $g$ is a Laurent polynomial with coefficients in $\Z$ and
whose Newton polytope has $\v 0$ as the unique lattice point in its interior. We define the power series 
\[
F(t)=\sum_{n\ge0}g_nt^n,\quad g_n=\mbox{constant term in }g(\v x)^n.
\]
We also define its truncated versions $F_m(t)=\sum_{n=0}^{m-1}g_nt^n$. 
Taking $\mu=\Delta^\circ$, the interior of $\Delta$, the space $\Omega_f^{(1)}(\Delta^\circ)$
has rank 1 and is generated by $1/f$. The base ring $R$ is the $p$-adic completion of $\Z[t,1/F_p(t)]$.
Theorem~\ref{main-Part-I} states that
\be{beukers-vlasenko}
\cartier\left(\frac{1}{f}\right)\is \lambda(t)\frac{1}{f^\sigma}\mod{p\fil_1^\sigma}
\ee
for some $\lambda(t)\in R$.
In Part II we show that $\lambda(t)=F(t)/F(t^\sigma)$ and obtain the following result as an application.
\begin{theorem}
For any odd prime $p$ and any integers $m,s\ge1$ we have
\be{mellit-vlasenko}
\frac{F(t)}{F(t^\sigma)}\is\frac{F_{mp^s}(t)}{F_{mp^{s-1}}(t^\sigma)}\mod{p^s}.
\ee
\end{theorem}
This theorem was first proven (with $m=1$ and $\sigma(t)=t^p$) in \cite{MeVl16}.
In Part II we give a proof which works for any $m$ and $\sigma$, see~\cite[Theorem 3.2 and Corollary 4.4]{BeVl20II}. 
The prototype for such congruences was given by Dwork in~\cite{dwork69} for the
case of the hypergeometric function $F(t)=F(1/2,1/2;1|t)$. 

There are many papers in which some very special form of \eqref{mellit-vlasenko} 
is proven, but modulo $p^{2s}$ or even higher powers $p^{3s},p^{4s},\ldots$. Usually this is
with $m=s=1$ and a specialization of $t$. These strengthenings are often called
{\it supercongruences}, a name coined by Jan Stienstra in the 1980's.
Such supercongruences arise only in very special situations
and not much is known about a general mechanism behind them. The original
motivation for the present paper was to discover such a mechanism.
For this purpose we will work in  $\Omega_f$ modulo $\fil_k$ with $k\ge2$. As we explain in Section~\ref{higher-derivatives}, elements $\omega \in \fil_k$ can be characterized by certain divisibility properties for the coefficients of their formal Laurent series expansion $\omega= \sum_{\v u} a_\v u \v x^\v u$. Namely, for each $\v u$ the coefficients $a_\v u$ is divisible by $g.c.d.(u_1,\ldots,u_n)^k$. For this reason we called $\fil_k$ the module of \emph{$k$th formal  derivatives}. Having an identity in $\Omega_f/\fil_k$, one obtains congruences by reading particular expansion coefficients on both sides. 

In Section \ref{symmetric-CY} we consider an application to $f=1-tg(\v x)$ where
$g$ are special Laurent polynomials with coefficients in $\Z$ and a large symmetry group. 

Our second main result is Theorem \ref{excellent-example2} which states
that there is a Frobenius lift $\sigma$ such that \eqref{beukers-vlasenko} 
with $\lambda(t)=F(t)/F(t^\sigma)$ holds modulo $p^2\fil_2$ instead of $p\fil_1$.
Any other Frobenius lift gives us far less elegant results. We call this special
Frobenius lift the {\it excellent lift}, a name already coined by Dwork. 
It is possible to derive supercongruences modulo $p^{2s}$ for the Laurent series expansion
coefficients of $1/f$ using Proposition~\ref{expansion-coeffs-k}. Unfortunately, we were
unable to prove a mod $p^{2s}$ version of \eqref{mellit-vlasenko},
which was one of the original motivations for this paper. A more precise formulation
of what we believe is true, is stated in Conjecture \ref{congruence-mod-p2s}.

In Theorem \ref{excellent-is-analytic} we show that the excellent Frobenius lift can be
written as a $p$-adic limit of rational functions in $t$ with powers of a prescribed polynomial in their denominators.
This has already been done by Dwork in \cite[\S7-\S8]{dwork69} in the context of families of elliptic curves,
where this is called Deligne's theorem. It is interesting to see
that this approximation property also holds outside the context of elliptic curves.

Although our study has not yet yielded the desired supercongruences, we did encounter
an application of our work to the problem of $p$-integrality of so-called instanton
numbers. This is a question which arose in the famous work of Candelas, De la Ossa, Green
and Parkes in their 1991 study of mirror symmetry within string theory. We
developed this application in a separate paper, see \cite{BeVlInstanton}.

{\bf Acknowledgement}. We are very grateful to the referee, who carefully read the article
and provided us with crucial comments. In particular the referee contributed ideas
towards the formulation of the proof of the crucial Theorem \ref{free-quotient-k}
and the referee's 
criticism on the example sections induced a complete rewrite of these sections.

\section{Dwork crystals}\label{sec:dwork-crystals}
In this section we define submodules of $\Omega_f$ whose $p$-adic completions are preserved
by the Cartier operation in a reasonable sense. An example is given by modules
$\Omega_f(\mu)\subset\Omega_f$ mentioned in the
introduction. Recall that $\mu$ is an open subset of $\Delta$ in the rudimentary topology on
$\Delta$ and $\Omega_f(\mu)$ consists of elements $(k-1)!A(\v x)/f(\v x)^k\in\Omega_f$
whose numerators A(\v x) are supported in $k\mu$ with $k\ge1$. Now
we would like to refine this construction by imposing further restrictions on polynomials
in the numerators. 

We remind the reader that $R$ is a $p$-adically complete ring with a fixed Frobenius lift $\sigma$.
For a Laurent polynomial $A(\v x)=\sum a_{\v u}\v x^{\v u}$, the polynomial
$\sum \sigma(a_{\v u})\v x^{\v u}$ is denoted by $A^\sigma$.
Our Cartier operator $\cartier$ acts on polynomials by the formula
\[
\cartier\left(\sum a_{\v u}\v x^{\v u}\right)=\sum a_{p\v u}\v x^{\v u}.
\]

\begin{definition}\label{cartier-f-compatible}

Let $L$ be an $R$-submodule of the module of Laurent polynomials with support in $\R_{\ge0}\Delta$
and coefficients in $R$. We denote by $L^\sigma$ the $R$-module generated by the elements $A^\sigma$ for $A\in L$.  We say that $L$ is \emph{($\sigma, f$)-compatible} if it satisfies the following assumptions.

\begin{itemize}
\item[(a)] $L$ is \emph{$p$-saturated}, which means that if $A(\v x)\in L$ and all coefficients of $A$ are
divisible by $p$, then $A(\v x)/p\in L$. The same holds for $L^{\sigma^i}$ for all $i \ge 1$.
\item[(b)] $f(\v x)L\subset L$.
\item[(c)] $\cartier(L)\subset L^\sigma$.
\end{itemize}
\end{definition}

With a small effort one can show
\begin{lemma}
If $L$ is ($\sigma,f$)-compatible, then $L^\sigma$ is ($\sigma,f^\sigma$)-compatible.
\end{lemma}

\begin{example}[\bf{Calabi-Yau families}]\label{example-admissible}
A very important example in our work is given by $f(\v x)=1-tg(\v x)$, where $g$
is a Laurent polynomial with coefficients in $\Z$ and a \emph{reflexive} Newton
polytope $\Delta$. In this paper we call the family of varieties $1-tg(\v x)=0$ a \emph{ Calabi-Yau family}.
The base ring $R$ will be either $\Z_p\lb t \rb$ or the $p$-adic completion of a ring of the form
$\Z_p[t,1/P(t)]$, where $P(t)\in\Z_p[t]$ with $P(0)\in\Z_p^\times$. The latter ring consists of so-called
$p$-adic analytic elements which appear in Dwork's work. Note that it can be embedded in $\Z_p\pow t$
using the power series expansion of $1/P(t)$.
\end{example}

We remind the reader that a lattice polytope $\Delta\subset\R^n$
is called {\it reflexive} if it is of maximal dimension and each of its codimension
1 faces can be given by an equation $\sum_{i=1}^na_ix_i=1$ 
with coefficients $a_i\in\Z$. It follows from this definition that
$\v 0$ is the unique lattice point in the interior $\Delta^\circ$.
Another property of reflexive polytopes is that
to every lattice point $\v u$ there is a unique integer $k\ge1$ such that
$\v u$ lies on the boundary of $k\Delta$.
This integer is denoted by $\deg(\v u)$, the degree of $\v u$. Let us remark that reflexivity of
the polytope implies that $\deg(p\v u)=p\deg(\v u)$. For a polynomial $A(\v x)$ we define its
degree $\deg(A)$ as the maximum of degrees of points in its support.

\begin{lemma}\label{CY-module}
Let $1-tg(\v x)=0$ be a Calabi-Yau family as in Example~\ref{example-admissible}. 
For the Frobenius lift $\sigma$ we assume that $t^\sigma/t^p\in R$.
Let $\Gamma\subset\Z^n$ be the lattice generated
by the support of $g$. Suppose $p$ does not divide $[\Z^n:\Gamma]$. Let $L$ be the $R$-module
of Laurent polynomials generated by $t^{\deg(\v u)}\v x^{\v u}$ with $\v u\in\Gamma$.
Then $L$ is ($\sigma,f$)-compatible. 
\end{lemma}

\begin{definition}\label{admissible-module}
The elements of the module $L$ defined in Lemma \ref{CY-module} are called \emph{ admissible polynomials}. 
\end{definition}

\begin{proof}[Proof of Lemma \ref{CY-module}]
Property (a) of Definition \ref{cartier-f-compatible} is obvious.
Since $f$ is admissible and the product of admissible polynomials is again
admissible, property (b) is also
clear. It suffices to verify property (c) on the generators $t^{\deg(\v u)}\v x^{\v u}$.
If $\v u$ is not divisible by $p$, its Cartier image is $0$. If $\v u=p\v v$,
then $\cartier(t^{\deg(p\v v)}\v x^{p\v v})=t^{p\deg(\v v)}\v x^{\v v}$. Up to
a factor in $R$ this equals $(t^\sigma)^{\deg(\v v)}\v x^{\v v}\in L^\sigma$. 
Indeed, since $R$ is $p$-adically complete the element $t^\sigma/t^p \in 1 + p R$ is invertible.
Furthermore, since $p$ does not divide $[\Z^n:\Gamma]$ the vector $\v v$ is again in $\Gamma$.
\end{proof}

\begin{definition}\label{dwork-crystal-pre-def}
Let $L$ be a ($\sigma,f$)-compatible $R$-module of Laurent polynomials and 
$\mu\subset\Delta$ an open set. The $R$-submodule of $\Omega_f$
generated by 
\[
(m-1)!\frac{A(\v x)}{f(\v x)^m},\quad \mbox{with $A\in L$ and ${\rm Supp}(A)\subset m\mu$}
\]
is denoted by $\Omega_{L,f}(\mu)$. Its $p$-adic completion is denoted by
$\hat\Omega_{L,f}(\mu)$.
\end{definition}

Let us now recall the Cartier operation $\cartier: \hat\Omega_f \to \hat \Omega_{f^\sigma}$ from Part I.
What we need is the following explicit formula for its action on rational functions which can be extracted from
the proof of \cite[Prop 3.3]{BeVl20I}. Let $G(\v x)$ be the polynomial defined by $pG(\v x)=f^\sigma(\v x^p)-f(\v x)^p$.
Then for any $m \ge 1$ and a polynomial $A(\v x)$ supported in $m \Delta$ the Cartier image of the respective
rational function is given by
\be{cartier-action}
\cartier\left((m-1)!\frac{A(\v x)}{f(\v x)^m}\right)=\sum_{r=0}^\infty\frac{p^r}{r!}
\times \frac{(m-1)!}{(\ceil{m/p}-1)!}\times(r+\ceil{m/p}-1)!
\frac{Q_r(\v x)}{f^\sigma(\v x)^{r+\ceil{m/p}}},
\ee
where
\[
Q_r(\v x)=\cartier(A(\v x)f(\v x)^{p\ceil{m/p}-m}G(\v x)^r).
\] 

\begin{proposition}\label{cartier-stable}
Let the assumptions be as in Definition~\ref{dwork-crystal-pre-def}.
The operator $\cartier$ maps $\hat\Omega_{L,f}(\mu)$ to $\hat\Omega_{L^\sigma,f^\sigma}(\mu)$.
\end{proposition}

\begin{proof}
Let $A \in L$ and $Supp(A) \subset m \mu$. We will show that each term $Q_r$ in~\eqref{cartier-action}
belongs to $L^\sigma$ and has support in $(\ceil{m/p}+r)\mu$. 
Multiply $Q_r$ with $p^r$ and observe that $p^rG(\v x)^r$ is an integer
linear combination of products $f(\v x)^{pr_1}
f^\sigma(\v x^p)^{r_2}$ with $r_1+r_2=r$. Notice that
\[
\cartier(A(\v x)f(\v x)^{p\ceil{m/p}-m}f(\v x)^{pr_1}f^\sigma(\v x^p)^{r_2})
=f^\sigma(\v x)^{r_2}\cartier(A(\v x)f(\v x)^{p\ceil{m/p}-m+ p r_1}).
\]
By $(\sigma,f)$-compatibility of $L$ we see that the last argument of $\cartier$
lies in $L$ and, since $\mu$ is open, it is supported in $(p\ceil{m/p}+pr_1)\mu$. Its $\cartier$-image
lies in $L^\sigma$ with support in $(\ceil{m/p}+r_1)\mu$. Since $L^\sigma$ is
($\sigma,f^\sigma$)-compatible, we see that after multiplication by $f^\sigma(\v x)^{r_2}$
we get an element in $L^\sigma$ with support in $(\ceil{m/p}+r_1+r_2)\mu=(\ceil{m/p}+r)\mu$.
So we find that $p^rQ_r$ lies in $L^\sigma$. Since $L^\sigma$ is $p$-saturated
we conclude that $Q_r\in L^\sigma$. 
\end{proof}

\begin{definition}[{\bf Dwork crystal and level submodules}]\label{dwork-crystal-def}
A \emph{ Dwork crystal} is an $R$-module of the form $M=\hat\Omega_{L,f}(\mu)$,
where $L$ is a ($\sigma,f$)-compatible $R$-module of Laurent polynomials
and $\mu\subset\Delta$ is an open subset. 

The Dwork crystal $\hat\Omega_{L^\sigma,f^\sigma}(\mu)$ is denoted by $M^\sigma$, and Proposition
\ref{cartier-stable} states that the Cartier operation $\cartier:\hat\Omega_f\to
\hat\Omega_{f^\sigma}$ restricts to $\cartier:M\to M^\sigma$.

For a fixed $k\ge1$
the $R$-module of functions $(k-1)!\frac{A(\v x)}{f(\v x)^k}\in M$ is called the part of $M$ of level $k$,
or simply the \emph{$k$-part} of $M$, and denoted by $M(k)$.
\end{definition}

Note that when $L$ is simply the $R$-module of Laurent polynomials we recover the crystals
$M=\hat\Omega_f(\mu)$ that were studied in Parts I and II. Our main
application concerns Calabi-Yau families. 

\begin{definition}[\bf Calabi-Yau crystal]\label{CY-crystal}
Let $f(\v x)=1-t g(\v x)$ be a polynomial as in Example~\ref{example-admissible} with a reflexive Newton polytope $\Delta$.
Let $L \subset R[x_1^{\pm 1},\ldots,x_n^{\pm 1}]$ be the respective module of admissible Laurent polynomials.
Then the Dwork module $M=\hat\Omega_{L,f}(\Delta^\circ)$ is
called a \emph{ Calabi-Yau crystal} (or \emph{CY-crystal}).
\end{definition}

\begin{remark}\label{free-level-k-part}
Let $p\ge k$. Then there is a natural inclusion 
$\iota_k:M(k-1)\to M(k)$
given by 
\[
\iota_k:\ (k-2)!\frac{A(\v x)}{f^{k-1}}\mapsto\frac{1}{k-1}(k-1)!
\frac{A(\v x)f}{f^{k}}.
\]
Note that for all examples of $L$ mentioned above the modules $M(k)$ are free.
\end{remark}

\section{Higher formal derivatives}\label{higher-derivatives}

In this section we discuss the filtration $\fil_k$ which was mentioned in the Introduction. It is defined as follows.

\begin{definition}\label{definition-fil}
Let $M$ be a Dwork crystal. We define
\[
\fil_k M:=\{\omega\in M|
\cartier^s(\omega)\in p^{sk}M^{\sigma^s}\}
\mbox{ for all $s\ge1$}.
\]
\end{definition}

We denote $\fil_k \hat\Omega_f$ by $\fil_k$. It is clear that $\fil_kM = \fil_k \cap M$. We will now see that the elements of $\fil_k$ can be characterized by divisibility properties of coefficients in their formal Laurent expansion. 

In \S 2 of Part I we expanded rational functions $\omega=A(\v x)/f(\v x)^m$ as a formal Laurent series with respect to
a vertex $\v b\in\Delta$. When $Supp(A) \subset m \Delta$, the result of such expansion is supported in the positive cone spanned by the shifted polytope $\Delta-\v b$:
\[
\omega = \sum_{\v u \in C(\Delta - \v b) \cap \Z^n} a_\v u \v x^\v u.
\]
The expansion coefficients $a_\v u = a_\v u(\omega)$ lie in the extension of the ring $R$ by $f_{\v b}^{-1}$,
a multiplicative inverse of the coefficient of $f(\v x)$ at the vertex $\v b$. For simplicity,
let us assume that $f_\v b$ is a unit in $R$. Then the procedure of formal expansion at $\v b$
defines an embedding of $\hat\Omega_f$ into the ring of formal series supported in the cone:
\[
\Omega_{\rm formal} = \{ \sum_{\v u \in C(\Delta - \v b)  \cap \Z^n} a_\v u \v x^\v u \;|\; a_\v u \in R  \}.
\] 
We then considered the Cartier operation $\cartier: \Omega_{\rm formal} \to \Omega_{\rm formal}$ given by 
$\sum a_\v u \v x^\v u \to \sum a_{p \v u} \v x^\v u$.
It was shown in Part I that this operation restricts to an $R$-linear map 
$\cartier: \hat \Omega_f \to \hat \Omega_{f^\sigma}$. This restricted map can be described by
formula~\eqref{cartier-action}, which shows its independence of the choice of the expansion vertex~$\v b$.


Denote the $R$-module of formal Laurent series by $\Omega_{\rm formal}$.
We assumed that the coefficient of $\v x^{\v b}$ in $f$ is a unit in $R$. 

\begin{definition}
Denote by $d^k\Omega_{\rm formal}$ the $R$-module generated by elements of the
form $\theta_{i_1}\cdots\theta_{i_k}\eta$, where $\theta_i=x_i\frac{\partial}{\partial x_i}$
and $\eta\in\Omega_{\rm formal}$. We call it the module of formal $k$-th derivatives. 
\end{definition}

We have the following analogue of \cite[Lemma 2.2]{BeVl20I} (Katz's lemma).

\begin{lemma}\label{katzlemma}
For a series $\eta = \sum a_\v u \v x^\v u \in\Omega_{\rm formal}$ the following conditions are equivalent:
\begin{itemize}
\item[(i)] $\eta \in d^k\Omega_{\rm formal}$
\item[(ii)] $a_\v u \in g.c.d.(u_1,\ldots,u_n)^k R$ for each $\v u \in C(\Delta-\v b)\cap \Z^n$; 
\item[(iii)] $\cartier^s(\eta)\is0\mod{p^{ks}\Omega_{\rm formal}}$ for all integers $s\ge1$.
\end{itemize}
\end{lemma}
\begin{proof}  (i) immediately implies (ii). We will prove that (ii) implies (i) by induction on $k$.
The case $k=0$ is obvious with the convention $d^0\Omega_{\rm formal}=\Omega_{\rm formal}$.
Let $k \ge 1$ and the claim is true for $k-1$. For each $\v u \in C(\Delta - \v b) \cap \Z^n$ we
choose some integers $k_i(\v u)$, $i=1,\ldots,n$ so that $g.c.d.(u_1,\ldots,u_n) = \sum_{i=1}^n k_i(\v u) u_i$ and
$b_\v u \in R$ such that $a_\v u = g.c.d.(u_1,\ldots,u_n) b_\v u$. 
Define $\eta_i = \sum k_i(\v u) b_{\v u} \v x^\v u$ for $i=1,\ldots,n$. 
By the inductional assumption we have $\eta_i \in d^{k-1}\Omega_{\rm formal}$.
Since $\eta=\sum_{i=1}^n \theta_i(\eta_i)$, we conclude that $\eta \in d^k\Omega_{\rm formal}$.

Equivalence of (ii) and (iii) is clear because $R$ is a $\Z_p$-algebra and all primes other
than $p$ are invertible in $\Z_p$.
\end{proof}

\begin{corollary}\label{formal-2-fil-k}
With the notations as above we have
\[
\fil_k M=M\cap d^k\Omega_{\rm formal}.
\]
\end{corollary}

For this reason we call $\fil_k M$ the \emph{submodule of formal $k$-th derivatives} in $M$. 
Note that the definition of $d^k \Omega_{\rm formal}$ depends on the expansion vertex $\v b$,
but Definition~\ref{definition-fil} shows that the module $\fil_k$ is independent of this choice.
In the case of Calabi-Yau crystals with $f(\v x)=1-t g(\v x)$ the coefficients of formal
expansion belong to the bigger ring $R[t^{-1}]$, and one could extend the above discussion to this case.
However, one can proceed slightly differently 
by looking at the formal expansions at $\v 0$. Such expansions are given by
\[
\frac{A(\v x)}{f(\v x)^m} = A(\v x) \sum_{k=0}^{\infty} \binom{k+m-1}{m-1} t^m g(\v x)^m = \sum_{\v u \in \Z^n} c_\v u \v x^\v u, 
\]
where $c_\v u \in \Z_p\lb t \rb$. The above series converges $t$-adically. On such expansions our Cartier operation again acts by $\cartier \left( \sum_{\v u \in \Z^n} c_\v u \v x^\v u \right) = \sum_{\v u \in \Z^n} c_{p \v u} \v x^\v u$. This was explained in~\cite[\S2]{BeVl20II}. One particular consequence of this formula is that for $\omega = \sum_{\v u \in \Z^n} c_\v u \v x^\v u \in \fil_k$ one has $c_\v u \in g.c.d.(u_1,\ldots,u_n)^k\Z_p\lb t \rb$ for every $\v u \in \Z^n$. We will exploit these congruences in Section~\ref{symmetric-CY}.

Finally let us mention the following direct consequence of Definition \ref{definition-fil}.

\begin{corollary}\label{cartier-div-by-k-on-fil-k}
The operator $\cartier$ sends $\fil_k M$ to 
$p^k\fil_k M^\sigma$.
\end{corollary}

We suspect that the rank of $M/\fil_k$ is always finite,
but we shall only be able to show this under conditions which are parallel to the
invertibility of the Hasse-Witt matrix in the case $k=1$. This will be the first main result of this paper. 

\section{The main theorem}\label{sec:main-thm}
Throughout this section we let $M$ be a Dwork crystal as in Definition~\ref{dwork-crystal-def}. 
Suppose that $p>k$. We determine the image of $\cartier$ modulo $p^k$ using \eqref{cartier-action}. 
Observe that
\begin{eqnarray*}
\ord_p\left(\frac{p^r}{r!}\frac{(m-1)!}{(\ceil{m/p}-1)!}\right)&>&r-\frac{r}{p-1}+\ceil{m/p}-1\\
&\ge&(r+\ceil{m/p}-1)\left(1-\frac{1}{p-1}\right)\\
&\ge&(r+\ceil{m/p}-1)\left(1-\frac{1}{k}\right).
\end{eqnarray*}
So whenever $r+\ceil{m/p}\ge k+1$ this order is $>k(1-1/k)=k-1$. Hence this order is $\ge k$.
Consequently, the terms in \eqref{cartier-action} will be $0$ modulo $p^k$ whenever $r+\ceil{m/p}\ge k+1$.
Therefore we conclude that
\be{cartier-image-mod-p-k}
\cartier(M)\subset M^\sigma(k)+p^k M^\sigma. 
\ee

\begin{definition}\label{maximal-cartier-image}
We will say that $\cartier$ is maximal on $M(k)$ if
\begin{itemize}
\item[(i)] To every $\omega'\in M^\sigma(k)$ there exists $\omega\in M(k)$
such that $\cartier(\omega)\is p^{k-1}\omega'\mod{p^k M^\sigma}$.
\item[(ii)] If $\omega\in M(k)$ and $\cartier(\omega)\is0\mod{p^k M^\sigma}$ then $\omega\in pM(k)$.
\end{itemize}
\end{definition}

Notice that when $k=1$, maximality of $\cartier$ on $M(1)$ comes down to
$\cartier:M(1)\mod{p}\to M^\sigma(1)\mod{p}$ being an isomorphism. This is the pivotal assumption in Theorem~\ref{main-Part-I} as we mentioned in the Introduction. In Section~\ref{sec:HW-matrices} we will give a simple criterion for the maximality of 
$\cartier$ on $M(k)$, see Proposition~\ref{hwc-maximal}. It is in terms of a generalization of
the Hasse-Witt determinant condition for $k=1$. Our criterion yields maximality of of $\cartier$ on $M^{\sigma^i}(k)$ for all $i$ simultaneously.

\begin{theorem}\label{free-quotient-k}
Let $M$ be a Dwork crystal. Suppose that $p>k$ and that
$\cartier$ is maximal on $M^{\sigma^i}(\ell)$ for all $i\ge0$ and $\ell=1,2,\ldots,k$. Then 
\be{k-decomp}
M\cong M(k)\oplus\fil_k M.
\ee
\end{theorem}

\begin{proof}
We are going to prove the decomposition~\eqref{k-decomp} simultaneously for all $M^{\sigma^i}$
and use induction on $k$.
The case $k=0$, when we agree that $\fil_0 M=M$ and $M(0)$ is zero,
is trivial.

Suppose now that $k>0$ and our claim is true for $k-1$. We thus have
\be{induction-assum}
M^{\sigma^i}\cong M^{\sigma^i}(k-1)\oplus \fil_{k-1}M^{\sigma^i} \text{ for all }\; i\ge 0.
\ee
We will show that 
\be{fil-split-0}
\fil_{k-1}M\cong\fil_{k-1}M(k)\oplus \fil_k M.
\ee
Here $\fil_{k-1}M(k)=\fil_{k-1}M\cap M(k)$, the formal $k-1$-st derivatives
in $M(k)$. 
As a consequence of the induction hypothesis and \eqref{fil-split-0} we get
\[
M\cong M(k-1)\oplus\fil_{k-1}M\cong
M(k-1)\oplus\fil_{k-1}M(k)\oplus\fil_k M.
\]
Restriction of \eqref{induction-assum} to $M(k)$ yields
$M(k)\cong M(k-1)\oplus\fil_{k-1}M(k)$ (as $R$-modules) 
and hence we find that $M\cong M(k)\oplus\fil_k M$, as desired.

We divide the proof of~\eqref{fil-split-0} in 4 steps in which we study the action of the operator
$\phi_p:=p^{1-k}\cartier$. Since $\cartier(x)\in p^{k-1}\fil_{k-1}M^\sigma$
for every $x\in\fil_{k-1}M$, we see that $\phi_p$ maps $\fil_{k-1}M$ to $\fil_{k-1}M^\sigma$.

{\bf Step 1}. The map $\phi_p$ maps $\fil_{k-1}M$ to $\fil_{k-1}M^\sigma(k)+p\fil_{k-1}M^\sigma$. Analogous statements hold for $M^{\sigma^i}$ for all $i\ge0$.
Let $x\in\fil_{k-1}M$. According to 
\eqref{cartier-image-mod-p-k}, we have $\cartier(x)\in M^\sigma(k)+p^kM^\sigma$. Using
the decomposition $M^\sigma\cong M^\sigma(k-1)\oplus\fil_{k-1}M^\sigma$ given in \eqref{induction-assum} this implies  
$\cartier(x)\in M^\sigma(k)+p^k\fil_{k-1}M^\sigma$. Using the observation $\cartier(x)\in p^{k-1}\fil_{k-1}M^\sigma$
we deduce that $\cartier(x)\in p^{k-1}\fil_{k-1}M^\sigma(k)+p^k\fil_{k-1}M^\sigma$.
Step 1 follows after division by $p^{k-1}$.

{\bf Step 2}. 
The map $\phi_p:\fil_{k-1}M(k)\to \fil_{k-1}M^\sigma(k)\mod{p\fil_{k-1}M^\sigma}$ is surjective. 
A similar statement holds for $M^{\sigma^i}$ for all $i\ge0$. 

Let $x\in\fil_{k-1}M^\sigma(k)$. Then according to Definition \ref{maximal-cartier-image}(i) there
exists $y_0\in M(k)$ such that $\cartier(y_0)\is p^{k-1}x\mod{p^kM^\sigma}$. Using decomposition~\eqref{induction-assum} in $M^\sigma$ we find $x_0\in M^\sigma(k-1)$ such that
\[
\cartier(y_0) \is p^{k-1}x + p^k x_0 \mod {p^k \fil_{k-1}M^\sigma}.
\]
We now construct by induction in $s \ge 0$ a sequence of elements $y_s \in M(k)$ and $x_s \in M^\sigma(k-1)$ such that
\be{step-2-induction}
\cartier\left(\sum_{i=0}^s p^i y_i\right) \is p^{k-1}x + p^{k+s} x_s \mod {p^k \fil_{k-1}M^\sigma}.
\ee
The case $s=0$ is already done. Suppose that~\eqref{step-2-induction} holds for some $s$. By a similar argument as for $s=0$ there exist $y_{s+1}\in M(k)$ and $x_{s+1}\in M^\sigma(k-1)$ such that
\[
\cartier(y_{s+1}) \is -p^{k-1}x_{s} + p^k x_{s+1} \mod {p^k \fil_{k-1}M^\sigma}.
\]
Multiplying this identity by $p^{s+1}$ and add to~\eqref{step-2-induction}, we obtain the desired identity for $s+1$.  

Using \eqref{induction-assum} there exists $x_1\in M^\sigma(k-1)$ such that $\cartier(y_0)\is p^{k-1}x+p^kx_1\mod{p^k\fil_{k-1}M^\sigma}$.
Setting $y=\sum_{i\ge0}p^iy_i$ we have $\cartier(y)\is p^{k-1}x\mod{p^k\fil_{k-1}M^\sigma}$.
Since $x\in\fil_{k-1}M^\sigma$ we see that $y\in\fil_{k-1}M$. By construction we know $y\in M(k)$,
hence $y\in\fil_{k-1}M(k)$. After division by $p^{k-1}$ we obtain $\phi_p(y)\is x\mod{p\fil_{k-1}M^\sigma}$,
which finishes the proof of Step 2.

{\bf Step 3}.
We have $\fil_{k-1}M=\fil_{k-1}M(k)+\fil_k M$ and similar statements for $M^{\sigma^i}$. 

Take $x \in \sF_{k-1}M$. Let us construct a Cauchy sequence $y_s\in\fil_{k-1}M(k)$
for $s\ge 0$ such that
\begin{itemize}
\item[(a)] $y_{s+1}-y_s\in p^{s}\fil_{k-1}M(k)$
\item[(b)] $\phi_p^s(x-y_s)\is 0\mod{p^s\fil_{k-1}M^{\sigma^s}}$.
\end{itemize}
We use induction on $s$. Take $y_0=0$. Then we see that the case $s=0$ becomes trivial.

Now suppose that $s\ge0$ and that (b) is proven for this value of $s$. Apply 
$\phi_p$ to $p^{-s}\phi_p^s(x-y_s)$. Then Step 1 shows the existence of
$\eta\in\fil_{k-1}M^{\sigma^{s+1}}(k)$ such that 
\[
p^{-s}\phi_p^{s+1}(x-y_s)\is \eta\mod{p\fil_{k-1}M^{\sigma^{s+1}}}.
\]
Step 2 (surjectivity of $\phi_p$, hence $\phi_p^{s+1}$) shows the existence of $\delta\in\fil_{k-1}M(k)$ such that
$\phi_p^{s+1}(\delta)\is\eta\mod{p\fil_{k-1}M^{\sigma^{s+1}}}$. Hence 
$\phi_p^{s+1}(p^s\delta)\is p^s\eta\mod{p^{s+1}\fil_{k-1}M^{\sigma^{s+1}}}$ and
\[
\phi_p^{s+1}(x-y_s-p^s\delta)\is 0\mod{p^{s+1}\fil_{k-1}M^{\sigma^{s+1}}}.
\]
So by taking $y_{s+1}=y_s+p^s\delta$ our induction step is completed. 
Let $y=\lim_{s\to\infty}y_s$. Then 
\[
\phi_p^s(x-y)\is\phi_p^s(x-y_s)\is0\mod{p^s\fil_{k-1}M^{\sigma^s}}
\]
for all $s\ge0$.
Hence $\cartier^s(x-y)\is0\mod{p^{sk}\fil_{k-1}M^{\sigma^s}}$ for all $s\ge1$. We conclude that 
$x-y\in\fil_kM$.

{\bf Step 4}. We show that $\fil_{k-1}M(k)\cap\fil_kM=\{0\}$. 

We first remark that if $x\in\fil_{k-1}M(k)$
and $x/p\in M(k)$, then $x/p\in\fil_{k-1}M(k)$. Namely, by our induction
hypothesis there exist unique $x_1\in M(k-1)$ and $x_2\in\fil_{k-1}M$
such that $x/p=x_1+x_2$. Hence $x-px_2=px_1$ with
$x-px_2\in\fil_{k-1}M$ and $px_1\in M(k-1)$. Since $M(k-1)\cap\fil_{k-1}M=\{0\}$ we
conclude that $px_1=0$, hence $x_1=0$. So $x/p=x_2\in\fil_{k-1}M(k)$.

The second remark is that property (ii) of Definition \ref{maximal-cartier-image} implies that
\[
\phi_p:\fil_{k-1}M(k)\mod{p\fil_{k-1}M}\to \fil_{k-1}M^\sigma(k)\mod{p\fil_{k-1}M^\sigma}
\]
is injective. Together with Step 2 we see that this map is an isomorphism.

Suppose we have non-zero $x\in\fil_{k-1}M(k)\cap\fil_kM$.
Let $s$ be the maximal integer such that $p^{-s}x\in M(k)$. By our first
remark we have that $p^{-s}x\in\fil_{k-1}M(k)$. Apply
$\phi_p^{s+1}$ to $p^{-s}x$. Because $x\in\fil_kM$ we get
\[
\phi_p^{s+1}(p^{-s}x)\is0\mod{p\fil_kM^{\sigma^{s+1}}}.
\]
Since $\phi_p^{s+1}$ is an isomorphism mod $p$ we find that $p^{-s}x$ is
divisible by $p$.
This contradicts our choice of $s$ and concludes our proof of Step 4.
\end{proof}

A consequence which will often be used is the following.

\begin{corollary}\label{decompose-cartier-image}
Let notations be as in Theorem \ref{free-quotient-k}. Then
\be{cartier-image-decomp}
\cartier(M) \subset M^\sigma(k) + p^k \fil_k M^\sigma.
\ee
\end{corollary}

\begin{proof}
From~\eqref{cartier-image-mod-p-k}  and Theorem \ref{free-quotient-k} we get:
\[\bal
\cartier(M) \subset M^\sigma(k) + p^{k} M^\sigma
&= M^\sigma(k) + p^{k} \left(M^\sigma(k) + \fil_{k}M^\sigma\right) \\
&= M^\sigma(k) + p^k \fil_{k} M^\sigma,
\eal\]
as asserted.
\end{proof}

\section{Hasse-Witt matrices}\label{sec:HW-matrices}
In order to use Theorem \ref{free-quotient-k} we shall require a practical criterion
to verify maximality of $\cartier$ on $M(k)$ in the sense of Definition
\ref{maximal-cartier-image}. Throughout this section we assume that $k<p$. In Part I we used the Hasse-Witt
matrix which is essentially the matrix corresponding to the Cartier map modulo $p$.
Formula \eqref{cartier-image-mod-p-k} now shows that $\cartier$ maps $M(k)$ to $M^\sigma(k)$
modulo $p^k$. In this section we will assume that the $M(k)$ are free modules.
This is the case in all our examples of Dwork crystals, see Remark \ref{free-level-k-part}.
Consider \eqref{cartier-action} with $m=k$. The estimate given at the beginning of 
Section~\ref{sec:main-thm} shows that the terms with $r \ge k$ on the right vanish  modulo $p^k$.
Rewriting the remaining terms we obtain that
\[ 
\cartier\left(\frac{A(\v x)}{f(\v x)^k}\right)\is
\frac{1}{f^\sigma(\v x)^k}\cartier\left(A(\v x)f(\v x)^{p-k}
\sum_{r=0}^{k-1}(f^\sigma(\v x^p)-f(\v x)^p)^r f^\sigma(\v x^p)^{k-r-1}\right)
 \mod{p^k M^\sigma}.
\]
Let us define
\be{definition-Fk}
F^{(k)}(\v x)=f(\v x)^{p-k}
\sum_{r=0}^{k-1}(f^\sigma(\v x^p)-f(\v x)^p)^r f^\sigma(\v x^p)^{k-r-1},
\ee
then
\[
\cartier\left(\frac{A(\v x)}{f(\v x)^k}\right)\is\frac{1}{f^\sigma(\v x)^k}
\cartier(A(\v x)F^{(k)}(\v x))\mod{p^k M^\sigma}.
\]
To write down a matrix for $\cartier\mod{p^k}$,
we assume that $M(k)$ has
a free basis $\frac{b_i(\v x)}{f(\v x)^k}, i=1,2,\ldots,m_k$ where $m_k$ is
the rank of $M(k)$.
We also assume that $\frac{b_i^\sigma(\v x)}{f^\sigma(\v x)^k}$ form a basis for $M^\sigma(k)$.

\begin{definition}\label{higherHasseWitt}
The $k$-th Hasse-Witt matrix is the $m_k\times m_k$-matrix $HW^{(k)}(M)$
with entries given by the formula
\[
HW^{(k)}(M)_{i,j}=\mbox{coefficient of $b_j^\sigma(\v x)$ in }
\cartier(b_i(\v x)F^{(k)}(\v x)). 
\]
\end{definition}

This definition makes sense due to the following lemma.
Recall that $M=\hat\Omega_{L,f}(\mu)$ for a $(\sigma,f)$-compatible submodule $L \subseteq R[x_1^{\pm 1},\ldots,x_n^{\pm 1}]$ and an open subset $\mu \subseteq \Delta$.

\begin{lemma} For $A \in L$ with $Supp(A) \subset k\mu$ the polynomial $\cartier(A F^{(k)})$ belongs to $L^\sigma$ and has support in $k \mu$. 
\end{lemma}
\begin{proof}
Note that $F^{(k)}$ is supported in $k(p-1)\Delta$. As $\mu$ is open, then $A F^{(k)}$ has support in $kp\mu$ and application of $\cartier$ yields a polynomial supported in $k \mu$. Expanding $F^{(k)}(\v x) = f(\v x)^{p-k}\sum_{r=0}^{k-1}\sum_{j=0}^r \binom{r}{j}(-1)^j f(\v x)^{pj} f^\sigma(\v x^p)^{k-j-1}$ on has
\[
\cartier( AF^{(k)}) = \sum_{r=0}^{k-1}\sum_{j=0}^r \binom{r}{j}(-1)^j f^\sigma(\v x)^{k-j-1} \cartier( A(\v x) f(\v x)^{pj}). \]
Eeach  $A(\v x) f^{pj}(\v x) \in L^\sigma$ by (b) in Definition~\ref{cartier-f-compatible} and $\cartier$ maps in to an element of $L^\sigma$. Recall that $L^\sigma$ is $(\sigma,f^\sigma)$-compatible, so every term remains in $L^\sigma$ after multiplication by a power of $f^\sigma$.   
\end{proof}

\begin{remark}
Note that Definition \ref{higherHasseWitt} depends on the choice of the basis vectors $b_i$.
This dependence is not made explicit in the notation since we shall be interested only in the
determinant of $HW^{(k)}(M)$ in our applications.

Strictly speaking, $HW^{(k)}(M)$ is the transpose of the matrix of $\cartier$ modulo $p^k$,
with the entries lifted to $R$. 
\end{remark}

In the case $M=\hat\Omega_f(\mu)$ we can use the basis $\frac{\v x^{\v u}}{f(\v x)^k}$
with $\v u\in k\mu$. The $k$-th Hasse-Witt matrix now has the entries
\[
HW^{(k)}(M)_{\v u,\v v}=\mbox{coefficient of $\v x^{p\v v-\v u}$ in }
F^{(k)}(\v x).
\]
In the case $k=1$ this is precisely the Hasse-Witt matrix from Parts I,II.
We then abbreviate $HW^{(k)}(\hat\Omega_f(\mu))$ by $HW^{(k)}(\mu)$. As a curiosity we mention that
\[
F^{(k)}(\v x) =  \frac{(f^\sigma(\v x^p)-f(\v x)^p)^k-(f^\sigma(\v x^p))^k}{-f(\v x)^k} \is \frac{f^\sigma(\v x^p)^k}{f(\v x)^k}\mod{p^k}.
\]
In order to determine the coefficient we have to consider a Laurent series expansion
on the right hand side.

In the case of a Calabi-Yau crystal, with $f=1-tg$, we use the basis 
$t^{\deg(\v u)}\frac{\v x^{\v u}}{f(\v x)^k}$ with $\v u\in k\Delta^\circ$. 
The associated Hasse-Witt matrix is then a twist of $HW^{(k)}(\Delta^\circ)$,
where the entries with indices $\v u,\v v$ are multiplied with 
$(t^\sigma)^{\deg(\v v)}/t^{\deg(\v u)}$.

By construction we have that
\be{HW-action}
\cartier\left(\frac{b_i}{f^k}\right)\is \sum_{j=1}^{m_k}HW^{(k)}_{i,j}(M)
\frac{b_j^\sigma}{(f^\sigma)^k}\mod{p^k M^\sigma}.
\ee

\begin{definition}[{\bf extended basis}]
Let $M$ be a Dwork crystal and $p>k$.
Denote the $R$-rank of $M(\ell)$ by $m_\ell$ for all $\ell$.
An extended basis of $M(k)$ is a basis $\omega_1,\ldots,\omega_{m_k}$
such that $\omega_1,\ldots,\omega_{m_\ell}$ is a free basis of $M(\ell)$
for $\ell=1,2,\ldots,k$. 
\end{definition}

The requirement of an extended basis for $M(k)$ is a mild one, as can be seen
from the following examples. The reader may wish to skip the overview of the
examples and proceed directly to Proposition \ref{hwc-maximal}. For a set $S \subset \R^n$ we denote $S \cap \Z^n$ by $S_\Z$.

\begin{lemma}\label{ext-basis-vertex-condition}
Suppose there is a vertex $\v b$ of $\Delta$ such that the coefficient of $\v x^{\v b}$ in
$f$ is a unit in $R$. Then, for any open set $\mu\subset\Delta$ and $p\ge k$ the set of functions
\[
(\ell-1)!\frac{\v x^{\v u}}{f(\v x)^\ell},\quad \v u\in (\ell\mu)_\Z\setminus(\v b+(\ell-1)\mu)_\Z,
\text{ when }\ell=2,\ldots,k,\quad \v u\in\mu_\Z\text{ when }\ell=1
\]
forms an extended basis of $\Omega_f(\mu)(k)$. 
\end{lemma}

\begin{proof}
In this proof we can assume, without loss of generality, that the coefficient of $\v x^{\v b}$ in $f$ is $1$.

When $k=1$ our lemma is obvious. Suppose $k>1$ and assume our lemma holds
for $\Omega_f(\mu)(k-1)$. Consider the standard basis of $\Omega_f(\mu)(k-1)$ embedded
in $\Omega_f(\mu)(k)$ via
\[
(k-2)!\frac{\v x^{\v u}}{f(\v x)^{k-1}}\mapsto (k-1)!\frac{\v x^{\v u}f(\v x)}{(k-1)f(\v x)^k}.
\]
We must show that this embedded set can be supplemented with the functions $(k-1)!\frac{\v x^{\v u}}{f^k}$ with
$\v u\in(k\mu)_\Z\setminus(\v b+(k-1)\mu)_\Z$ to a free basis of $\Omega_f(\mu)(k)$.
To that end we show that every Laurent polynomial $A$ with support in $k\mu$ can be
written in the form $A=Pf+Q$, where $P,Q$ are Laurent polynomials with support in $(k-1)\mu$,
$k\mu\setminus(\v b+(k-1)\mu)$ respectively. In our considerations we can ignore the factorials
$(k-1)!$ and $(k-2)!$ because they are units in $R$, we assumed that $p\ge k$.

Let $C(\Delta\setminus\v b)$ be the positive cone generated by the points of $\Delta$. Note that $\v 0$ is its unique
vertex. We impose a partial ordering on $k\mu$ by saying that $\v u$ is larger than $\v v$ if $\v u\ne\v v$
and $\v v\in\v u+C(\Delta\setminus\v b)$. We shall say that a finite subset $S\subset (k\mu)_\Z$ is closed with respect
to the ordering if $\v u\in S$ implies that $\v v\in S$ for every $\v v\in(k\mu)_\Z$ smaller than $\v u$. Notice
that this closedness property is preserved if we remove a maximal element from $S$. 

We now describe Euclidean division by $f$. Start with a Laurent polynomial $A$ with support
in $k\mu$. Set $P:=0,Q:=A$ and let $S\subset(k\mu)_\Z$ be such that it contains the support of $Q$
and is order closed. We perform the following loop.
\bigskip

{\bf Loop}: 
Choose a maximal element in $S$ and suppose it has the form $\v u+\v b$ with $\v u\in((k-1)\mu)_\Z$.
If no such element exists our procedure ends and $Q$ is the desired remainder of our division.
In the other case let $c\v x^{\v u+\v b}$ be the corresponding term and set
$P:=P+c\v x^{\v u},Q:=Q-cx^{\v u}f(\v x)$ and $S:=S\setminus\{\v b+\v u\}$.
\bigskip

Note that for any $\v w$ in the support of $f$ we have $\v u+\v w=\v u+\v b+\v w-\v b\in
\v u+\v b+C(\Delta\setminus\v b)$. Hence all new terms in $Q$ have support which is strictly smaller than $\v u+\v b$.
Therefore the new $S$ contains the support of the new $Q$.
Since $(k\mu)_\Z$ is finite, this loop terminates and after completion
we find that $A=Pf+Q$, where $P,Q$ have the desired form.
\end{proof}

Note that if the condition of Lemma~\ref{ext-basis-vertex-condition} is satisfied for $f$,
then it is also satisfied for $f^\sigma$.

\begin{lemma}\label{ext-basis-in-admissible-crystals} 
Let $M$ be a Calabi-Yau crystal as in Definition~\ref{CY-crystal} with $f(\v x)=1-t g(\v x)$ and $R=\Z_p\lb t \rb$.
Then the set 
\[
\frac{t^{\ell-1}\v x^\v u}{f(\v x)^\ell}, \v u\in\Delta^\circ\text{ with } \deg\v u=\ell-1
\text{ and }\ell=1,2,\ldots,k
\]
is an extended basis for $M(k)$.  
\end{lemma}
 
Note that this basis looks more natural than the one in Lemma \ref{ext-basis-vertex-condition}.
The down side is that Lemma \ref{ext-basis-in-admissible-crystals} yields a $\Z_p\pow t$-basis.
However the same set is a basis over a ring $R=\Z_p[t,1/P(t)]$ with a suitably chosen polynomial $P(t)$,
which depends of $f$ and $k$.
We shall not go into the details of this.

\begin{proof}[Proof of Lemma \ref{ext-basis-in-admissible-crystals}]
We show that any element of $M(k)$ can be written as a linear combination of
the above basis elements plus an element of $tM(k)$. By repeating this procedure we find that
any element of $M(k)$ lies in the $\Z_p\pow t$-span of the above basis vectors.

Consider a term $\frac{t^{\deg \v u}\v x^\v u}{f(\v x)^\ell}$ with $\deg(\v u)<\ell$ 
(because $\v u\in\ell\Delta^\circ$). We write it as 
\[
\frac{t^{\deg(\v u)}\v x^{\v u}(f+tg)^{\ell-1-\deg(\v u)}}{f^\ell}=\sum_{r=0}^{\ell-1-\deg(\v u)}
\binom{\ell-1-\deg(\v u)}{r}\frac{t^{\deg(\v u)+r}\v x^{\v u}g(\v x)^r}{f(\v x)^{r+1+\deg(\v u)}}.
\]
The terms in $t^{\deg(\v u)+r}\v x^{\v u}g(\v x)^r$ of degree equal to $\deg(\v u)+r$ give rise to a
linear combination of the above mentioned basis elements. The term with degree $<\deg(\v u)+r$
contribute to terms that are contained in $tM(k)$.
\end{proof}

Let us now give a practical criterion to verify the maximality of $\cartier$ on $M(k)$.

\begin{proposition}\label{hwc-maximal}
Let $p>k$ and let $M$ be a Dwork crystal such that $M(k)$ has an extended basis.
Let $m_\ell$ be the rank of $M(\ell)$ for $\ell\le k$.
Then $\det(HW^{(k)}(M))$ is divisible by $p^{L(k)}$,
where 
\[
L(k) = \sum_{\ell=1}^{k-1} (m_k-m_\ell).
\]
Moreover, if 
$\det(HW^{(k)}(M))$ divided by $p^{L(k)}$ is in $R^\times$
then $\cartier$ is maximal on $M(k)$ in the sense of
Definition \ref{maximal-cartier-image}.
\end{proposition}

\begin{proof}
Let $\omega_i, i=1,2,\ldots,m_k$ be an extended basis of $M(k)$.
Using this basis we rewrite \eqref{HW-action} as
\be{HW-action-split}
\cartier(\omega_i)\is \sum_{\ell=1}^k\sum_{j=m_{\ell-1}+1}^{m_\ell}HW^{(k)}_{i,j}(M)
\omega_j^\sigma\mod{p^k M}.
\ee
Because $\cartier(M)\subset M(r)+p^r M^\sigma$ for all $r\le k$,
we see that $HW^{(k)}_{i,j}$ is divisible by $p^{\ell-1}$ whenever $m_{\ell-1}<j$.
This implies that the $p$-adic order of $\det(HW^{(k)}(M))$ is at least
\[
\sum_{\ell=1}^{k}(\ell-1)(m_{\ell}-m_{\ell-1})=\sum_{\ell=1}^k (m_k-m_\ell),
\]
which equals $L(k)$. This proves the first statement.

Let $\tilde{HW}$ be the matrix which is obtained from $HW^{(k)}$ after division
of the $j$-th column by $p^{\ell-1}$ for all $j$ with $m_{\ell-1}<j\le m_{\ell}$ and $\ell=1,\ldots,k$. 
If $\det(HW^{(k)}(M))\in p^{L(k)}R^\times$ we
see that $\tilde{HW}$ is invertible. Equation \eqref{HW-action} becomes
\[
\cartier(\omega_i)\is\sum_{\ell=1}^{k-1}\sum_{j=m_{\ell-1}+1}^{m_\ell}\tilde{HW}_{ij}
p^{\ell-1}\omega_j^\sigma\mod{p^k M}.
\]
Since $\tilde{HW}$ is invertible, we see that modulo $p^k M$ all terms
$p^{\ell-1}\omega_j^\sigma$ with $m_{l-1}<j\le m_\ell$ are in the image of $\cartier$. Hence all
terms $p^{k-1}\omega_j^\sigma$ with $j=1,\ldots,m_k$ are in the image of $\cartier$ modulo $p^k$, and so
property (i) of Definition \ref{maximal-cartier-image} is satisfied.

Suppose on the other hand that we have $\lambda_i\in R$ such that
$\cartier(\sum_i\lambda_i\omega_i)\is0\mod{p^k M^\sigma}$. Hence
$\sum_i\lambda_i\tilde{HW}_{ij}p^{\ell-1}$ is divisible by $p^k$
for every $j,\ell$ with $m_{\ell-1}<j\le m_\ell$. In particular the sums $\sum_i\lambda_i\tilde{HW}_{ij}$
are divisible by $p$ for all $j$. By the invertibility of $\tilde{HW}$ we find that
all $\lambda_i$ must be divisible by $p$. Hence property (ii) of Definition
\ref{maximal-cartier-image} holds.
\end{proof}

\begin{definition}
The determinant of the Hasse-Witt matrix $HW^{(k)}(M)$ divided by $p^{L(k)}$ is called
the $k$-th Hasse-Witt determinant and is denoted by 
\[
hw^{(k)}(M):=p^{-L(k)}\det(HW^{(k)}(M)).
\]
\end{definition}

We note that $hw^{(k)}(M)$ depends on the choice of basis, though this is not made explicit in our notation. Using a different basis in $M(k)$ results in $hw^{(k)}(M)$ being multiplied by $u/u^\sigma \in R^\times$, where $u \in R^\times$ is the determinant of the change of basis matrix.
Notice also that if $hw^{(k)}(M)\in R^\times$, then the same
holds for $hw^{(k)}(M^{\sigma^i})$ for all $i\ge0$.
Then, as a consequence of Proposition
\ref{hwc-maximal}, $\cartier$ is maximal on all $M^{\sigma^i}(k)$.
We thus obtain the following version of Theorem \ref{free-quotient-k}.

\begin{corollary}\label{main-theorem-alt}
If $p>k$,
$M(k)$ has an extended basis and $hw^{(\ell)}(M)\in R^\times$ for $\ell=1,2,\ldots,k$,
then $M\cong M(k)\oplus \fil_k M$.
\end{corollary}

If the assumptions of this corollary is satisfied we will say that \emph{the $k$th Hasse-Witt condition holds for $M$}.
In practice
we compute the determinants $hw^{(\ell)}(M)$ 
for $\ell=1,\ldots,k$ and verify that 
they are not divisible by $p$. If that is
the case we extend $R$ to the $p$-adic completion
of $R[1/(hw^{(1)}\cdots hw^{(k)})]$.

\begin{remark}
In Adolphson-Sperber \cite[(2.3)]{AS17} the authors also define level $k$ Hasse-Witt matrices,
but in the particular case that the interior of $(k-1)\Delta$ contains no lattice
points. In that case the computation in \cite{AS17} 
seems to come down to the determination of the coefficient of $\v x^{p\v v-\v u}$ 
of $f^{k(p-1)}$, where $\v u,\v v\in(k\Delta^\circ)_\Z$.
\end{remark}

Let us suppose that the $k$th Hasse-Witt condition holds for $M$.
Similarly to Part I, we can define matrices of the Cartier operator
and connection on the free quotients 
\[
Q^{(k)}(M):=M/\fil_k M \cong R^h,
\]
where $h=m_k$ is the rank of $M(k)$. Derivations $\delta: R \to R$ extend to operations $\delta: \Omega_f \to \Omega_f$ which act on the coefficients of rational functions through the usual rules of differential calculus. We have $\delta(r \omega)=\delta(r)\omega+r\delta(\omega)$ for any $r \in R$ and $\omega\in \Omega_f$. 
Suppose that $\omega_1,\ldots,\omega_h$ is a basis of $Q^{(k)}(M)$.
We define 
the respective matrices $\Lambda = (\lambda_{ij}) \in R^{h \times h}$ and $ 
N_\delta = (\nu_{ij}) \in R^{h \times h}$ for derivations such that $\delta(M) \subset M$ by
\be{k-matrices-def}\bal
\cartier(\omega_i) &= \sum_j \lambda_{ij} \omega_j^\sigma 
\mod {p^k \fil_k M^\sigma}, \\
\delta(\omega_i) &= \sum_j \nu_{ij} \omega_j \mod {\fil_k M}.
\eal\ee 

An important application of Theorem \ref{free-quotient-k} is the existence
of so-called {\it supercongruences for the expansion coefficients}
of a function $\omega\in\Omega_f$. 
Choose a vertex $\v b \in \Delta$ and suppose that the coefficient $f_{\v b}$ of $\v x^{\v b}$ in
$f$ is not divisible by $p$. Take the Laurent series expansion of $\omega$ with support in
the positive cone $C(\Delta-\v b)$. For any $\v{m} \in C(\Delta - \v b)_\Z$ we denote 
\be{exp-coeffs-def}\alpha_{\v m}(\omega) :=
\text{ coefficient of $\v x^{\v m}$ in the expansion of } \omega \text{ at } \v b.
\ee
Notice that $\alpha_{\v m}(\omega) \in R[f_\v b^{-1}]$. See Section~\ref{higher-derivatives} and~\cite[\S 2]{BeVl20I} 
for more details regarding formal expansions. 
The following Proposition generalizes \cite[Thm 6.2]{Ka85} and the analogous
\cite[Thm 5.7]{BeVl20I}.

\begin{proposition}[{\bf supercongruences}]\label{expansion-coeffs-k} 
Fix $\v{m} \in C(\Delta - \v b)_\Z$.
Take notations and assumptions as above and let $s$ be a positive integer.
Then the vectors $\v a_{p^s \v m} \in R^h$ with components
\[
(\v a_{p^s \v m})_i = \alpha_{p^s \v m}(\omega_i) = 
\text{ coefficient of $\v x^{p^s \v m}$ in } \omega_i
\]
satisfy the congruences
\[\bal
\v a_{p^{s} \v m} & \equiv  \Lambda \; 
\v \sigma(\v a_{p^{s-1} \v m}) \mod {p^{sk}}, \\
\delta(\v a_{p^s \v m}) &\equiv  N_\delta \; \v a_{p^{s} \v m} \mod {p^{sk}}. \\
\eal\]
\end{proposition}

\begin{proof} (Similar to \cite[Thm 5.7]{BeVl20I} in Part I.)
Expand \eqref{k-matrices-def} in a Laurent series and take on both sides
the coefficient of $\v x^{p^{s-1}\v m}$. 
Use the fact that $\alpha_{p^{s-1} \v m}(\cartier(\omega_i)) =
\alpha_{p^{s} \v m}(\omega_i)$ and $\alpha_{p^{s-1}\v m}(\omega_j^\sigma)
=\alpha_{p^{s-1}\v m}(\omega_j)^\sigma$. 
For the second congruence take the coefficient on both sides of $\v x^{p^s\v m}$.
The claimed congruences follow because, 
by Lemma~\ref{katzlemma}, the coefficient of $\v x^{p^s\v m}$
of elements of $\fil_k$ are divisible by $p^{sk}$.  
\end{proof}

It is also possible to give a system of differential equations for $\Lambda$
as follows. 

\begin{proposition}\label{Cartier-connection-relation}
Let $\Lambda$ and $N_\delta$ be the matrices defined in \eqref{k-matrices-def}. Then
\[
N_\delta \Lambda = \Lambda N_\delta^{\sigma} + \delta(\Lambda),
\]
where $N_\delta^{\sigma}$ is the (transpose) matrix of 
$\delta$ on $Q^{(k)}(M^\sigma)$ in the basis $\omega_i^\sigma$.
\end{proposition}

\begin{proof}
Denote the column vector with entries $\omega_1,\ldots,\omega_h$
by $\boldomega$. Then \eqref{k-matrices-def} can be summarized as $\cartier(\boldomega)
=\Lambda\boldomega^\sigma$ and $\delta(\boldomega)=N_\delta\boldomega$.
Apply $\delta$ to $\cartier(\boldomega)=\Lambda\boldomega^\sigma$. We get
\[
\delta(\cartier(\boldomega))=\delta(\Lambda\boldomega^\sigma).
\]
Since $\delta\circ\cartier=\cartier\circ\delta$ the left hand side can be
rewritten as $\cartier(\delta(\boldomega))=\cartier(N_\delta\boldomega)=
N_\delta\Lambda\boldomega^\sigma$.
The right hand side can be rewritten as 
\[
\delta(\Lambda)\boldomega^\sigma+\Lambda\delta(\boldomega^\sigma)=
\delta(\Lambda)\boldomega^\sigma+\Lambda N_\delta^\sigma\boldomega^\sigma.
\]
Thus we find that 
\[
N_\delta\Lambda\boldomega^\sigma=\delta(\Lambda)\boldomega^\sigma+
\Lambda N_\delta^\sigma\boldomega^\sigma. 
\]
Since $\omega_1^\sigma,\ldots,\omega_h^\sigma$ are independent over $R$ we can
drop $\boldomega^\sigma$ from this equality and find the desired result.
\end{proof}

\begin{remark}
The same proposition, with the same proof, can also be shown
for the matrices $\Lambda_\sigma$ and $N_\delta$ in Theorem 5.3 of \cite{BeVl20I}.
This proposition forms the basis for the so-called Frobenius structure of the
system of linear differential equation associated to the matrices $N_\delta$
for all derivations $\delta$ of $R$. Examples will be worked out in the next
sections.
\end{remark}

\section{A simple example}\label{sec:example1}

Consider the Dwork crystal $\hat\Omega_f$ with 
$f(\v x) = (1-x_1)(1-x_2)-tx_1 x_2=1-x_1-x_2+(1-t)x_1x_2$. For the base ring $R$ we take
the $p$-adic completion of $\Z_p[t,1/t]$. We will see later that
the Hasse-Witt conditions require $t$ to be invertible in $R$. 
As Frobenius lift we shall take $\sigma:t\to t^p$. Later we briefly
deal with other Frobenius lifts. 

The goal of this section is to prove the following.
\begin{theorem}\label{theorem-example1}
For any prime $p\ge3$ and the Frobenius lift given by $t^\sigma = t^p$ the we have
\[
\cartier\left(\frac{1}{f}\right)\is\frac{1}{f^\sigma}\mod{p^2\fil_2^\sigma}.
\]
\end{theorem}
This theorem can be proven in a straightforward way by computation of the coefficients
of the power series expansion of $1/f$. However, we prefer to give another proof
which illustrates the methods from the previous sections. First we like to mention a
consequence.

\begin{corollary}\label{simple-example-specialization}
For any prime $p\ge3$ we have
\[
\cartier\left(\frac{1}{1-x_1-x_2+2x_1x_2}\right)\is\frac{1}{1-x_1-x_2+2x_1x_2}\mod{p^2\fil_2^\sigma}.
\]
\end{corollary}

\begin{proof}
Set $t=-1$ in Theorem \ref{theorem-example1}.
\end{proof}

Write $\frac{1}{1-x_1-x_2+2x_1x_2}=\sum_{k,l\ge0}\alpha_{k,l}x_1^kx_2^l$. 
Taking the coefficient at $x_1^{k p^{s-1}}x_2^{l p^{s-1}}$ in the identify from
Corollary~\ref{simple-example-specialization} we obtain that
\[
\alpha_{kp^s,lp^s}\is\alpha_{kp^{s-1},lp^{s-1}}\mod{p^{2s}}
\]
for all $k,l,s\ge0$. This is a special case of a more general phenomenon of 
supercongruences for Taylor series coefficients of rational functions in several
variables. A striking example is Armin Straub's discovery \cite{Straub14} of supercongruences 
modulo $p^{3s}$ for the Taylor coefficients of 
\[
\frac{1}{(1-x_1-x_2)(1-x_3-x_4)-x_1x_2x_3x_4}
\]
The diagonal coefficients $\alpha_{n,n,n,n}$ are the so-called Ap\'ery numbers.

We now prove Theorem \ref{theorem-example1}.
The Newton polytope $\Delta$ is the unit square. Let's work in the subcrystal $\hat\Omega_f(\mu)$ with  
\[
\mu = \Delta \setminus \{\text{upper edge $\cup$ right edge} \}.
\]
Then $(\mu)_\Z=\{(0,0)\}$ and 
\[
(2 \mu)_\Z = \Delta_\Z = \{ (0,0), (1,0), (0,1), (1,1) \}.
\]
Let us verify that the Hasse-Witt determinants $hw^{(1)}(\mu)$ and $hw^{(2)}(\mu)$ are invertible.
The Hasse-Witt matrix $HW^{(1)}(\mu)$ is a $1\times1$-matrix
with entry the constant coefficient of $f(\v x)^{p-1}$, which is 1.
An extended basis in the level 2 part of 
$\hat\Omega_f(\mu)$ is given by
$1/f,x_1/f^2,x_2/f^2,x_1x_2/f^2$. We compute $HW^{(2)}(\mu)$ with respect to this
basis. We have
\[
(p \v v - \v u)_{\v u,\v v \in (2 \mu)_\Z} = 
\begin{pmatrix} (0,0) & * & * & * \\ (-1,0) & (p-1,0) & * & * \\
(0,-1) & (p,-1) & (0,p-1)& * \\ (-1,-1) & (p-1,-1) & (-1,p-1)& (p-1,p-1) 
\end{pmatrix}.
\]
Since all terms under the diagonal have negative components, $HW^{(2)}(\mu)$
is an upper-triangular matrix. To read its diagonal entries,
it is enough to work modulo $(x_1^p,x_2^p)$. In particular we have 
$f^\sigma(\v x^p) \equiv 1$ and the polynomial \eqref{definition-Fk} becomes
\[
F^{(2)}(\v x)=f(\v x)^{p-2}(2 f^\sigma(\v x^p) - f(\v x)^p)
\equiv 2 f(\v x)^{p-2} - f(\v x)^{2p-2} \mod {(x_1^p,x_2^p)}.
\]
The diagonal entries of $HW^{(2)}(\mu)$ are given by
\[
1, - \binom{2p-2}{p-1}, - \binom{2p-2}{p-1}, 
- \binom{2p-2}{p-1} \sum_{m=0}^{p-1} \binom{p-1}{m}^2 (1-t)^m,  
\]
from which we get that 
\[
\binom{2p-2}{p-1}^{-3} \det HW^{(2)}(\mu) \equiv -\sum_{m=0}^{p-1} (1-t)^m
\equiv - t^{p-1} \mod p. 
\]
Note that $\binom{2p-2}{p-1}$ is exactly divisible by $p$. 
Since $t \in R^\times$, $hw^{(2)}(\mu)=\det HW^{(2)}(\mu)/p^3$ is invertible. 
So Theorem \ref{free-quotient-k} and Corollary~\ref{main-theorem-alt} tell us that
\[
\hat\Omega_f(\mu)=R(1/f)\oplus R(x_1/f^2)\oplus R(x_2/f^2)\oplus R(x_1x_2/f^2)\oplus\fil_2,
\]
where $\fil_2$ is short for $\fil_2\hat\Omega_f(\mu)$. Instead of $x_1 x_2/f^2$ we prefer to
use the related basis vector $tx_1x_2/f^2=\theta(1/f)$, where $\theta=t \frac{\partial}{\partial t}$.
A similar decomposition holds for $\hat\Omega_{f^\sigma}(\mu)$ and by Corollary~\ref{decompose-cartier-image} there
exist $\beta_1,\beta_2,\beta_3,\beta_4\in R$ such that
\be{cartier-example1}
\cartier(1/f)\is\beta_1/f^\sigma+\beta_2 x_1/(f^\sigma)^2+\beta_3 x_2/(f^\sigma)^2+
\beta_4 \theta(1/f)^\sigma \mod{p^2\fil_2^\sigma}
\ee
where $\fil_2^\sigma$ is short for $\fil_2\hat\Omega_{f^\sigma}(\mu)$. 
We can expand $1/f$ in a power series in $x_1,x_2$ and set $x_2=0$. We obtain
\[
\cartier\left(\frac{1}{1-x_1}\right)\is\beta_1 \frac{1}{1-x_1}+\beta_2 
\frac{x_1}{(1-x_1)^2}\mod{p^2\fil_2^\sigma}.
\]
Since we have trivially that $\cartier(\frac{1}{1-x})=\frac{1}{1-x}$ we conclude that
$\beta_1=1$ and $\beta_2=0$. Similarly, by setting $x_2=0$, we find that $\beta_3=0$.
The determination of $\beta_4$ is more subtle. One easily verifies that
\[
\theta^2\left(\frac1f\right)=x_1\frac{\partial}{\partial x_1}\left(x_2\frac{\partial}{\partial x_2}\left(t \frac{x_1 x_2}{f}\right)\right) \in \fil_2.
\]
We apply the operator $\theta^2$ to
\eqref{cartier-example1} with $\beta_1=1,\beta_2=\beta_3=0$. We use the relation
$\theta(\nu^\sigma)=p(\theta\nu)^\sigma$ for any rational function $\nu$ and get 
\[
\cartier(\theta^2(1/f))\is p^2\theta^2(1/f)^\sigma
+\theta^2(\beta_4)\theta(1/f)^\sigma+2p\theta(\beta_4)\theta^2(1/f)^\sigma
+p^2\beta_4\theta^3(1/f)^\sigma\mod{\fil_2^\sigma}.
\]
All terms that contain $\theta^2(1/f)^\sigma$ or $\theta^3(1/f)^\sigma$ are in $\fil_2^\sigma$.
Hence we are left with $\theta^2(\beta_4)\theta(1/f)^\sigma\in\fil_2^\sigma$. Since $\theta(1/f)^\sigma$
is a basis vector modulo $\fil_2$ we conclude that
\[
\theta^2\beta_4=0.
\]
Since $\beta_4$ is in the completion of $\Z_p[t,1/t]$, the only possibility is
that $\beta_4$ is a constant.
The remaining work now resides in the determination of this constant $\beta_4$.
As we will see later, the final determination of some
constants in the matrix of $\cartier$ is a recurring theme.
In all cases we perform this calculation by studying supercongruences of the type given
in Proposition \ref{expansion-coeffs-k}. This is a powered up version of Katz's original
idea in \cite{Ka85} of the internal reconstruction of a (unit root) crystal. 

Let $a_Q(t)$ be the
coefficient of $x_1^Qx_2^Q$ in the power series expansion  
\[
\frac1f=\sum_{r=0}^\infty(x_1+x_2+(t-1)x_1x_2)^r.
\]
Taking the coefficient of $x_1^{p^{s-1}}x_2^{p^{s-1}}$ in~\eqref{cartier-example1} we obtain congruence
\be{congruence-example1}
a_{p^s}(t)\is a_{p^{s-1}}(t^p)+\beta_4 (\theta a_{p^{s-1}})(t^p)\mod{p^{2s}}.
\ee

\begin{lemma} The coefficient $a_Q(t)$ is a polynomial of degree $Q$ in $t$ 
with leading term $t^Q$.
\end{lemma}

\begin{proof}
The coefficient of $x^Qy^Q$ in $(x+y+(t-1)xy)^r$ is $0$ if $r<Q$ and 
$\frac{r!}{(r-Q)!(r-Q)!(2Q-r)!}(t-1)^{2Q-r}$ when $r \ge Q$. The latter term 
is of degree $<Q$ when $r>Q$, which shows that $a_Q(t)$ has the desired form.
\end{proof}

As corollary we see that $\theta a_Q(t)$ has again degree $Q$ but highest
degree coefficient $Q$. Taking the coefficients of $t^{p^s}$ in 
\eqref{congruence-example1} we
find that
\[
1\is 1+\beta_4 p^{s-1}\mod{p^{2s}}.
\]
Hence $\beta_4\is0\mod{p^{s+1}}$ for all $s\ge1$ and thus we conclude $\beta_4=0$.
This proves Theorem \ref{theorem-example1}.

Finally we formulate a version of Theorem \ref{theorem-example1} for 
general Frobenius lifts $\sigma: R \to R$.
To that end we start with Theorem \ref{theorem-example1} and re-expand 
$\frac1f(t^p)$
in $\hat\Omega_{f^\sigma}$. For a function $h(t)$ one has the expansion
\[
h(b e^x)=\sum_{r\ge0}\frac{1}{r!}(\theta^r h)(b)x^r.
\]
We apply it to $h=1/f$, $b=t^\sigma$ and $x=\log(t^p/t^\sigma)\in R$. 
Since $t^\sigma/t^p \in 1 + pR$ then $\log(t^p/t^\sigma) = \sum_{m=1}^\infty 
(-1)^{m-1} (t^\sigma/t^p)^m / m \in pR$.  Hence
\[
\frac1f(t^p)=\sum_{r\ge0}\frac{1}{r!}\left(\theta^r \frac1f \right)(t^\sigma)\log(t^p/t^\sigma)^r
\]
is a $p$-adically converging series. Note that for $r\ge2$ the terms are in $p^2 \fil_2^\sigma$
because $\theta^2(1/f)\in\fil_2$ and $\ord_p(r!)\le r-2$. Thus we find that 
\[
\frac1f(t^p)\is \frac1f(t^\sigma)+\log(t^\sigma/t^p)\, \left(\theta \frac1f \right)(t^\sigma)\mod{p^2 \fil_2^\sigma}.
\]
Hence it follows from Theorem \ref{theorem-example1} that
\begin{corollary}
For any Frobenius lift $\sigma:R \to R$ we have
\[
\cartier\left( \frac1f\right) \is \frac1{f^\sigma} + \log(t^\sigma/t^p) \; \left(\theta \frac1f \right)^\sigma \;\mod{p^2 \fil_2^\sigma}.
\]
\end{corollary}
We thus see that the coefficients of the Cartier matrix modulo $\fil_2$ depend
on the choice of the Frobenius lift $\sigma$. Note that $t \mapsto t^p$ is the unique lift for which the above idenitity simplifies to  $\cartier(1/f) \is 1/f^\sigma \mod {p^2 \fil_2^\sigma}$. It reminds us of what
Dwork called an {\it excellent Frobenius lift}, see \cite[Section 2]{Dwork73}. We will see more examples of such special lifts in the following section.

\section{Completely symmetric Calabi-Yau families}\label{symmetric-CY}
In this section we shall deal with symmetric cases of Calabi-Yau families.
Recall Example~\ref{example-admissible} where a Calabi-Yau family is given by
$f(\v x)=1-tg(\v x)=0$, where $g\in\Z[x_1^{\pm1},\ldots,x_n^{\pm1}]$, and $g$ has a reflexive Newton polytope $\Delta$. From reflexivity of $\Delta$ it follows that $\v 0$ is the unique interior lattice point
in its interior $\Delta^\circ$. For $R$ we take a $p$-adically closed subring of $\Z_p\lb t \rb$, which will be specified later. On $R$ we fix a Frobenius lift $\sigma$ such that $t^\sigma \in t^p R$.

A monomial substitution is a substitution
of the form $x_i\to \v x^{\v m_i}, i=1,\ldots,n$, where $\v m_i\in\Z^n$ and 
$|\det(\v m_1,\ldots,\v m_n)|=1$. Let $\mathcal{G}$ be a finite group of monomial substitutions
that fix $g$ and such that $p$ does not divide $\#\mathcal{G}$. Recall from Definition~\ref{admissible-module} that polynomial $A = \sum a_\v u \v x^\v u$ with coefficients in $\Z_p\lb t \rb$ is called admissible if $\ord_t(a_\v u) \ge \deg(\v u)$ for each $\v u \in Supp(A)$. Then the module $L$ of admissible
Laurent polynomials that are fixed under $\mathcal{G}$ is ($\sigma,f$)-compatible.
To prove this we use the fact that $\cartier$ commutes with elements of $\mathcal{G}$.
We call $g$ {\it completely symmetric} if the only non-zero lattice points in $\Delta$ are vertices and $\mathcal{G}$ acts transitively on the
set of vertices. In particular, all non-constant terms of $g$ have the same coefficient,
 which we will call $\gamma$. The standing assumption is that $\gamma\in\Z_p^\times$. 

\begin{definition}\label{completely-symmetric}
Let $g(\v x)$ be completely symmetric and $L \subset R[x_1^{\pm1},\ldots,x_n^{\pm1}]$ be the $R$-module of $\mathcal{G}$-invariant admissible Laurent polynomials. Then $\hat\Omega_{L,f}(\Delta^\circ)$ is called a \emph{completely symmetric Calabi-Yau crystal}. 
We denote it by $CY(g)$.
\end{definition} 

Here are some typical examples of completely symmetric families $1-tg(\v x)=0$, named after the shape of
their Newton polytope.
\smallskip

{\bf Simplicial family}.
Let 
\[
g(\v x)=x_1+x_2+\cdots+x_n+\frac{1}{x_1\cdots x_n}.
\]
The symmetry group $\mathcal{G}$ is generated
by the permutations of $x_1,\ldots,x_n$ and the involution that maps $x_1$ to $1/(x_1\cdots x_n)$
and $x_i$ to itself for $i\ge2$. These families are quotients of the so-called Dwork families
which one usually finds in the literature. Consider the Dwork family given by the projective
equation $y_0^{n+1}+y_1^{n+1}+\cdots+y_n^{n+1}=\lambda y_0y_1\cdots y_n$. It has
a symmetry group consisting of substitutions $y_i\to\zeta_iy_i$ with
$\zeta_i^{n+1}=1$ for all $i$ and $\zeta_0\zeta_1\cdots\zeta_n=1$. The functions 
$x_i=y_i^{n+1}/(y_0y_1\cdots y_n)$ with $i=1,\ldots,n$ generate the field of invariants under this group.
One easily sees that $x_1+\cdots+x_n+\frac{1}{x_1\cdots x_n}=\lambda$.
Letting $\lambda=1/t$ we recover our simplicial family.
\smallskip

{\bf Hypercubic family}.
Let 
\[
g(\v x)=\left(x_1+\frac{1}{x_1}\right)\left(x_2+\frac{1}{x_2}\right)\cdots\left(x_n+\frac{1}{x_n}\right).
\]
The group $\mathcal{G}$ is now generated by the permutations of $x_1,\ldots,x_n$ and the involution
that maps $x_1$ to $1/x_1$ and $x_i$ to itself for $i\ge2$. Notice that the lattice spanned by the
exponent vectors of $g(\v x)$ consists of $(k_1,\ldots,k_n)\in\Z^n$ where $k_1,\ldots,k_n$ are all even
or all odd. This is a sublattice of index $2^{n-1}$ in $\Z^n$.
\smallskip

{\bf Hyperoctahedral family}.
Let
\[
g(\v x)=x_1+\frac{1}{x_1}+x_2+\frac{1}{x_2}+\cdots+x_n+\frac{1}{x_n}.
\]
The group $\mathcal{G}$ is the same as in the previous example, but the lattice spanned by the
support of $g$ is $\Z^n$. 
\smallskip

{\bf \An-family}.
Let 
\[
g(\v x)=(1+x_1+\cdots+x_n)\left(1+\frac{1}{x_1}+\cdots+\frac{1}{x_n}\right).
\]
The symmetry becomes manifest if we introduce homogeneous variables $y_0,y_1,\ldots,y_n$
and set $x_i=y_i/y_0$. We then obtain
\[
g=(y_0+y_1+\cdots+y_n)\left(\frac{1}{y_0}+\frac{1}{y_1}+\cdots+\frac{1}{y_n}\right),
\]
The symmetry in the $y_i$ is obvious. If we interchange
$y_i$ and $y_j$ with $i,j\ne0$, the corresponding affine variables $x_i$ and $x_j$ interchange
as well. If we interchange $y_0$ and $y_i$ with $i\ne0$, the affine variables $x_j$ undergo the
monomial transformation $x_i\to1/x_i$ and $x_j\to x_j/x_i$ for all $j\ne i$.
The group $\mathcal{G}$ is the generated by these tranformations together with the
involution $x_i\to1/x_i,i=1,\ldots,n$.
The lattice spanned by the support of $g$ in the $x_i$ is simply $\Z^n$.
The name of the family derives from the observation that the support of the homogeneous $g$ is 
the root system $A_n$. When $n=4$
the family is also called the Hulek-Verrill family, which is extensively studied in \cite{COES19}.
\smallskip

Consider the power series
\[
F(t):=\sum_{n\ge0}g_nt^n, \quad g_n  =\text{constant term of }g(\v x)^n.  
\]
A straightforward calculation shows that in the case of a simplicial family we have
\[
F(t)={}_nF_{n-1}\left(\left.\tfrac{1}{n+1},\tfrac{2}{n+1}\ldots,\tfrac{n}{n+1};1,\ldots,1\right|((n+1)t)^{n+1}\right),
\]
where ${}_nF_{n-1}$ is one variable hypergeometric function of order $n$. It satisfies
the differential equation
\[
\mathcal{P}F=0\text{ with }\mathcal{P}=-\theta^n+((n+1)t)^{n+1}(\theta+1)\cdots(\theta+n),
\]
where $\theta=t\frac{d}{dt}$. 

Another straightforward computation shows that in the hypercubic example we have
\[
F(t)={}_nF_{n-1}\left(\left.\tfrac{1}{2},\tfrac{1}{2}\ldots,\tfrac{1}{2};1,\ldots,1\right|4^nt^2\right).
\]
It satisfies the differential equation
\[
\mathcal{P}F=0\text{ with }\mathcal{P}=-\theta^n+4^nt^2(\theta+1)^n. 
\]
The last two families are not of hypergeometric type. When $n=4$, the hyperoctahedral family
corrsponds to \#16 in the database of Calabi-Yau equations \cite{AESZ10} with $z=t^2$. 
When $n=4$, the {\bf \An}-family corresponds to \#34 in \cite{AESZ10}.
In the discussion following Proposition
\ref{second-order-de} we shall display a more complete list of formulas. 

We choose a completely symmetric $g(\v x)$. 
Let us compute the Hasse-Witt determinants.
The part of level~1 in $CY(g)$ is spanned by $1/f$, and due to the complete symmetry
condition the part of level~2 has $1/f^2$ and
$tg/f^2$ as a basis. It is clear from
Definition~\ref{higherHasseWitt} that the first Hasse-Witt determinant $hw^{(1)}(t)$ is 
independent of $\sigma$. Moreover, we have
\[\bal
hw^{(1)}(t) = \text{ constant term of } f(\v x)^{p-1} &=
\sum_{n=0}^{p-1} \binom{p-1}{n} (-1)^n g_n t^n \\
&\is  \sum_{n=0}^{p-1} g_n t^n \quad \mod p. 
\eal\]
Since $R$ is $p$-adically complete, we note that an element is invertible in $R$
if and only if it is invertible modulo $p$. Suppose that $hw^{(1)} \in R^\times$.
Then Theorem~\ref{main-Part-I} implies that $CY(g) = R \frac1 f \oplus \fil_1$ and
we have $\cartier(1/f) = \lambda / f^\sigma \mod {p \fil_1^\sigma}$ 
for some $\lambda \in R$. Expanding both sides as formal series at $\v 0$ and taking
coefficients at $\v x^\v 0$, we find that $F(t)=\lambda F(t^\sigma)$.  
Therefore we have congruence

\be{congruence-mod-fil1}
\cartier\left(\frac{1}{f(\v x)}\right)\is \frac{F(t)}{F(t^\sigma)}\frac{1}{f^\sigma(\v x)}
\mod{p\fil_1^\sigma}.
\ee
This fact was essentially proved in~\cite[Corollary 3.1]{BeVl20II}, where
we have used the Frobenius lift $t\to t^p$, but the proof
runs equally well with any Frobenius lift. A main objective of this section is to show
that for a special Frobenius lift $\sigma$ congruence~\eqref{congruence-mod-fil1} holds
modulo $p^2 \fil_2^\sigma$. We first make an observation regarding the second
Hasse-Witt determinant.

\begin{lemma} Modulo $p$ the determinant $hw^{(2)}(t)$ is a polynomial in $t$ which
is independent of the lift $\sigma$. Moreover, $hw^{(2)}(0) \in \Z_p^\times$.
\end{lemma}

\begin{proof} Recall that 
\[
F^{(2)}(\v x) = 2 f(\v x)^{p-2} f^\sigma(\v x^p) - f(\v x)^{2p-2}
\]
and the entries of the second Hasse-Witt matrix $HW^{(2)}$ in the basis $1/f^2$, $tg/f^2$ 
are given by 
\begin{eqnarray*}
HW^{(2)}_{00}&=&\mbox{const term of }F^{(2)}(\v x)\\
HW^{(2)}_{10}&=&\mbox{const term of }tg(\v x)F^{(2)}(\v x)\\
HW^{(2)}_{01}&=&\mbox{coefficient of }t^\sigma g(\v x)\mbox{ in }\cartier(F^{(2)}(\v x))\\
HW^{(2)}_{11}&=&\mbox{coefficient of }t^\sigma g(\v x)\mbox{ in }\cartier(tg(\v x)F^{(2)}(\v x)).
\end{eqnarray*}
From Proposition \ref{hwc-maximal} we know that $p$ divides $\det(HW^{(2)})$ and by 
construction we see that $hw^{(2)}(t)$ is a polynomial in $t,t^\sigma$ and $t^p/t^\sigma$.
So modulo $p$ the polynomial $hw^{(2)}(t)$ is independent of the choice of $\sigma$.

Let us now proceed to show that $hw^{(2)}(0)\in\Z_p^\times$ using the particular lift
$t^\sigma=t^p$. 
Notice that immediately $HW^{(2)}_{00}|_{t=0}=1$ and $HW^{(2)}_{10}|_{t=0}=0$.
In order to determine $HW^{(2)}_{11}$ we might as well determine the coefficient
of any non-constant term in $\cartier(tg(\v x)F^{(2)}(\v x))$, say $t^p\v x^{\v b}$,
where $\v b$ is a vertex of $\Delta$. So we need to
determine the coefficient of $t^p \v x^{p\v b}$ in $tg(\v x)F^{(2)}(\v x)$
Let us expand modulo $t^{p+1}$,
\[
tg(\v x)F^{(2)}(\v x)\is -2\sum_{r=0}^{p-2}{p-2\choose r}(-tg(\v x))^{r+1}
+\sum_{r=0}^{p-1}{2p-2\choose r}(-tg(\v x))^{r+1}\mod{t^{p+1}}.
\]
The first summation doesn't contribute anything to the term with $\v x^{p\v b}$.
The only term in the second summation that contributes to $\v x^{p\v b}$ is
${2p-2\choose p-1}(-tg(\v x))^p$. The only term in this $p$-th power that contributes
to $\v x^{p\v b}$ is the term with $(\v x^{\v b})^p$. Hence the desired coefficient is
${2p-2\choose p-1}(-t\gamma)^p/t^p\is-\gamma^p{2p-2\choose p-1}\mod{t}$,
where $\gamma$ is the coefficient of $\v x^{\v b}$ in $g(\v x)$. We are given
that $\gamma\in\Z_p^\times$. So we see that
\[
HW^{(2)}=\begin{pmatrix} 1+O(t) & O(t)\\ * & -\gamma{2p-2\choose p-1}+O(t)\end{pmatrix}.
\]
We get $\det(HW^{(2)})=-\gamma{2p-2\choose p-1}+O(t)$.
From Proposition \ref{hwc-maximal}
it follows that $p$ divides $\det(HW^{(2)})$. Moreover, $p$ divides ${2p-2\choose p-1}$
exactly once and thus we see that $hw^{(2)}=\det(HW^{(2)})/p$ is a polynomial in $\Z_p[t]$
and $hw^{(2)}(0)\in\Z_p^\times$.  
\end{proof}

We are now ready to state the main result of this section. For this we note that $p$-adic completion of $\Z_p[t, 1/hw^{(1)}(t),1/hw^{(2)}(t)]$ is independent of the choice of the Frobenius lift $\sigma$ in the definition of the Hasse-Witt determinants $hw^{(i)}$. This is because polynomials $hw^{(1)}$ and $hw^{(2)}$ modulo $p$ do not depend on $\sigma$. Moreover, this completion is a subring of $\Z_p\lb t \rb$ because $hw^{(1)}(0), hw^{(2)}(0)
\in \Z_p^\times$.

\begin{theorem}\label{excellent-example2}
Suppose $g(\v x)$ is completely symmetric. Let $p$ be an odd prime that does not divide
$\gamma\times\#\mathcal{G}\times|\Z^n:\Gamma]$, where $\Gamma$ is the lattice spanned by the support of $g(\v x)$
and $\gamma$ the vertex coefficient of $g$. 
Put $f(\v x)=1-t g(\v x)$ and let $CY(g)$ be the corresponding Calabi-Yau crystal as defined in Definition \ref{completely-symmetric}. Assume that the Hasse-Witt determinants $hw^{(1)}(t)$ and $hw^{(2)}(t)$ are invertible in the base ring $R \subseteq \Z_p \lb t \rb$. 
Then there exists a unique Frobenius lift $\sigma$ on $R$ such that
\be{congruence-mod-fil2}
\cartier\left(\frac{1}{f(\v x)}\right)\is \frac{F(t)}{F(t^\sigma)}\frac{1}{f^\sigma(\v x)}\mod{p^2\fil_2^\sigma}.
\ee
Moreover, for this Frobenius lift the element $t^\sigma$ belongs to the $p$-adic completion of the ring $\Z_p[t, 1/hw^{(1)}(t),1/hw^{(2)}(t)]$.

\end{theorem}

The special Frobenius lift in Theorem \ref{excellent-example2} is called the
{\it excellent Frobenius lift}. In what follows we will be able to describe it explicitly.
\begin{remark}
Note that a form of Theorem \ref{excellent-example2} in the context of the Legendre family 
of elliptic curves and in a very different language also occurs in 
Dwork's $p$-adic cycles, \cite[Thm 8.1]{dwork69}.
\end{remark}

In \cite[(7)]{BeVl20II} we give the following consequence of \eqref{congruence-mod-fil1}.
For every $m,s\ge1$ we have
\be{congruence-mod-ps}
\frac{F(t)}{F(t^\sigma)}\is \frac{F_{mp^s}(t)}{F_{mp^{s-1}}(t^\sigma)}\mod{p^s}.
\ee
Here $F_N(t)=\sum_{n=0}^{N-1}f_nt^n$ is the truncation of $F(t)$ at $t^N$. Our second main
objective was to show that \eqref{congruence-mod-ps} holds modulo $p^{2s}$ when 
$t\mapsto t^\sigma$
is an excellent Frobenius lift. Unfortunately we were unable to prove this for any example of $g$.
Numerical experiment seems to support the following conjecture.

\begin{conjecture}\label{congruence-mod-p2s}
Congruence \eqref{congruence-mod-ps} holds modulo $p^{2s}$ for $m=2$ in the hypercubic and hyper-octahedral families
and $m=n+1$ in the simplicial families.
\end{conjecture}

The rest of this section is dedicated to the proof of Theorem~\ref{excellent-example2}. We work under the assumptions of the theorem. We fix a Frobenius lift $\sigma: R \to R$ such that $t^\sigma = t^p v$ with $v \in 1 + p R$. Since the first and second Hasse-Witt determinants for $CY(g)$ are invertible in the base ring $R$, by Corollary~\ref{main-theorem-alt} we have the decomposition 
\be{CY-level-2-decomp}
CY(g) \cong CY(g)(2) \oplus \fil_2.
\ee
The level~2 part $CY(g)(2)$ is a free $R$-module with the basis $1/f$ and $1/f^2$. Consider the derivation $\theta = t \frac{d}{dt}$. It will be convenient for us to work in the basis $1/f$ and $\theta(1/f) = tg(\v x)/f^2(\v x)= 1/f^2-1/f$.

\begin{proposition}\label{second-order-de}
There exist $A,B\in R$ be such that $\theta^2(1/f)\is A(t)(1/f)+B(t)\theta(1/f)\mod{\fil_2}$.
Moreover, $A(0)=B(0)=0$. 
\end{proposition}

\begin{proof}
The first statement is an immediate consequence of 
the decomposition~\eqref{CY-level-2-decomp}. 
To prove the second statement we expand $1/f$ in a Laurent series in $x_1,\ldots,x_n$ by
\[
\frac{1}{f}=\sum_{k\ge0}t^kg(\v x)^k
\]
and then termwise expanion of the powers $g^k$. This will give a Laurent series with support
in $\Gamma$, the lattice generated by the support of $g$.
However, the series converges $t$-adically.
Choose a vertex $\v v$ of $\Delta$ and let $\gamma$ be the coefficient of $\v x^{\v v}$ in $g$.
By assumption we have $\gamma\in\Z_p^\times$. 

Let $Q$ be any positive integer and $a_Q(t)$ the coefficient of $\v x^{Q\v v}$
in the Laurent series
for $1/f$. Clearly $a_Q(t)$ is a power series in $t$. We also see that its lowest
degree term comes
from the term $t^Qg(\v x)^Q$. Its contribution to $a_Q(t)$ is given by the coefficient of
$\v x^{Q\v v}$
in $t^Qg^Q$, which is $\gamma^Qt^Q$. Hence $a_Q(t)=\gamma^Qt^Q(1+O(t))$.
Take $Q=p^s$ for any $s\ge1$.
Then $\theta^2(1/f)\is A(t)/f+B(t)\theta(1/f)\mod{\fil_2}$ implies that
\[
\theta^2a_{p^s}(t)\is A(t)a_{p^s}(t)+B(t)\theta a_{p^s}(t)\mod{p^{2s}}.
\]
Take the coefficient of $t^{p^s}$ on both sides. Then we get $0\is A(0)+B(0)p^s\mod{p^{2s}}$.
So $A(0)\is0\mod{p^s}$ for all $s\ge1$, hence $A(0)= 0$. Using this we also conclude that
$B(0)\is0\mod{p^s}$ for all $s\ge1$, hence $B(0)=0$.
\end{proof}

Surprisingly, the power series expansions of $A(t),B(t)$ are independent of the
choice of $p$. This is shown
in Theorem \ref{second-order-solutions} by displaying two independent solutions of the
differential equation $\theta^2y-B(t)\theta y-A(t)y=0$ which do not depend on $p$. As a result, 
$A(t),B(t)$ are also independent of $p$.
One solution is the power series $F(t)$ that we introduced earlier.
The other solution has the form
$F(t)\log t+G(t)$, where $G(t)$ is a power series in $t$ and $G(0)=0$. 

From Theorem \ref{second-order-solutions} we find that we must determine the lattice of integer
relations $\ell_1\v v_1+\cdots+\ell_N\v v_N=\v 0$, where $\v v_1,\ldots,\v v_N$ are the vertices
of $\Delta$. We list the results for our four families here.
\smallskip

{\bf Simplicial family}. We have $n+1$ vertices and a rank one lattice of relations
generated by $(1,1,\ldots,1)$. It immediately follows from Theorem \ref{second-order-solutions} that
\[
F(t)=\sum_{k\ge0}\frac{((n+1)k)!}{(k!)^{n+1}} t^{(n+1)k}
\]
and
\[
G(t)=\sum_{k\ge1}\frac{((n+1)k)!}{(k!)^{n+1}}
\times\sum_{j=k+1}^{(n+1)k}\frac{1}{j}\times t^{(n+1)k}.
\]
\smallskip

{\bf Hypercubic family}.
This is not so easy to deduce from Theorem \ref{second-order-solutions}. However
we can a take short cut in the proof of Theorem \ref{second-order-solutions}.
Consider the expansion
\[
\frac{1}{f}=\sum_{k\ge0}t^k\prod_{i=1}^n\left(x_i+\frac{1}{x_i}\right)^k
\]
and determine the coefficient of $(x_1\cdots x_n)^Q$ for any non-negative integer $Q$. A straightforward
calculation gives us
\[
\sum_{k\ge0,k\is Q\mod2}\binom{k}{\frac{k+Q}{2}}^nt^k.
\]
Replace $k$ by $2k+Q$. We get
\[
\sum_{k\ge0}\binom{2k+Q}{k}^nt^{Q+2k}.
\]
The proof of Theorem \ref{second-order-solutions} tells us that we get $F(t)$ by setting $Q=0$. Hence
\[
F(t)=\sum_{k\ge0}\left(\frac{(2k)!}{(k!)^2}\right)^n t^{2k}.
\]
We get $G(t)$ by taking the derivative with respect to $Q$ and then setting $Q=0$. Hence
\[
G(t)=\sum_{k\ge1}\left(\frac{(2k)!}{(k!)^2}\right)^n
\times\sum_{j=k+1}^{2k}\frac{n}{j}\times t^{2k}.
\]
\smallskip

{\bf Hyperoctahedral family}.
A straightforward computation shows that the coefficient of $x_1^Q$ in the $t$-expansion of $1/f$ equals
\[
\sum_{k\ge0}t^{Q+2k}\sum_{k_1+k_2+\ldots+k_n=k}\frac{(Q+2k)!}{k_1!(k_1+Q)!(k_2!\cdots k_n!)^2}.
\]
Setting $Q=0$ yields
\[
F(t)=\sum_{k\ge0}t^{2k}\sum_{k_1+k_2+\cdots+k_n=k}\frac{(2k)!}{(k_1!k_2!\cdots k_n!)^2}.
\]
Taking the $Q$-derivative and then $Q=0$ yields
\[
G(t)=\sum_{k\ge1}t^{2k}\sum_{k_1+k_2+\cdots+k_n=k}\left(\sum_{j=k_1+1}^{2k}\frac{1}{j}\right)
\frac{(2k)!}{(k_1!k_2!\cdots k_n!)^2}.
\]
\smallskip

{\bf \An-family}. 
A straightforward computation shows that the coefficient of $x_1^Q$ ($(y_1/y_0)^Q$ in the homogeneous
notation) in the $t$-expansion of $1/f$ equals
\[
\sum_{k\ge0}t^{Q+k}\sum_{k_0+k_1+\ldots+k_n=k}\frac{(Q+k)!^2}{k_0!(k_0+Q)!k_1!(k_1+Q)!(k_2!\cdots k_n!)^2}.
\]
Setting $Q=0$ yields
\[
F(t)=\sum_{k\ge0}t^k\sum_{k_0+\cdots+k_n=k}\left(\frac{k!}{k_0!\cdots k_n!}\right)^2.
\]
Taking the $Q$-derivative and then $Q=0$ yields
\[
G(t)=2\sum_{k\ge1}t^k\sum_{k_0+\cdots+k_n=k}\left(\frac{k!}{k_0!\cdots k_n!}\right)^2
\sum_{j=k_0+1}^{k}\frac{1}{j}.
\]
\smallskip

We now turn to the action of $\cartier$ on $CY(g)/\fil_2$. Recall that
$t^\sigma = t^p v$ with $v \in 1 + p R$.

\begin{proposition}\label{cartier-on-example}
Let notations and assumptions be as in Theorem \ref{excellent-example2}. 
Then there exist $\lambda_0(t),\lambda_1(t)\in R$ such that 
\be{cartier-on-rank2}
\cartier(1/f)\is \lambda_0(t)(1/f^\sigma)+\lambda_1(t)\theta(1/f)^\sigma\mod{p^2\fil_2^\sigma}.
\ee
Furthermore, $\lambda_0(0)=1$, $\lambda_1(t)$ is divisible by $p$ and 
$\lambda_1(0)=\log(\gamma^{p-1}/v(0))$.
\end{proposition}

\begin{proof}
The first statement follows from Corollary \ref{decompose-cartier-image} applied to our particular
situation. From \eqref{cartier-image-mod-p-k} with $k=1$ we conclude that for some
$\lambda \in R$ one has $\cartier(1/f) \is \lambda/f^\sigma \mod p$.
Since $1/f^\sigma$ and $\theta(1/f)^\sigma$ are linearly independent over $R$,
it follows that {$\lambda_0 \equiv \lambda \mod p$ and} $\lambda_1(t)$
is divisible by $p$.
To show that $\lambda_0(0)=1$ and $\lambda_1(0)=0$, consider the same expansion of
$1/f$ at $\v 0$ as in the proof
of Proposition \ref{second-order-de}. We shall again denote by $a_Q(t)$ the coefficient at
$\v x^{Q\v v}$ in this expansion.
Here $\v v$ is a vertex of the Newton polytope of $g$. In that proof we saw that 
$a_Q(t)=\gamma^{Q} t^Q(1+O(t))$,
where $\gamma\in\Z_p^\times$ is the vertex coefficient of $g(\v x)$. Let $Q=p^s$ for any $s\ge1$.
Expanding at $\v 0$ both sides of~\eqref{cartier-on-rank2} and taking the coefficients at
$\v x^{p^{s-1}\v v}$, we obtain that
\[
a_{p^s}(t)\is \lambda_0(t)a_{p^{s-1}}(t^\sigma)+\lambda_1(t)(\theta a_{p^{s-1}})(t^\sigma)\mod{p^{2s}}.
\]
Taking the coefficient of $t^{p^s}$ on both sides of the last congruence we get

\[
\gamma^{p^s} \is \lambda_0(0)\gamma^{p^{s-1}}v(0)^{p^{s-1}} + \lambda_1(0)p^{s-1}\gamma^{p^{s-1}}v(0)^{p^{s-1}} \mod {p^{2s}}.
\]
We rewrite this congruence as 
\[
\left( \frac{\gamma^{p-1}}{v(0)}\right)^{p^{s-1}} \is \lambda_0(0)+\lambda_1(0)p^{s-1}\mod{p^{2s}}
\]
Since $\gamma^{p-1}/v(0) \is 1 \mod p$, the left-hand side $\is 1 \mod{p^s}$. Therefore $\lambda_0(0)=1$. Subtracting $1$ and dividing by $p^{s-1}$, we find that $\lambda_1(0)=\log(\gamma^{p-1}/v(0))$.
\end{proof}

We like to determine the coefficients $\lambda_0(t),\lambda_1(t)$ of the Cartier
operator and use Proposition \ref{Cartier-connection-relation} in our case,
which has rank $2$. We take $\delta=\theta$ and note that
\[
N_\theta=\begin{pmatrix}0 & 1\\ A(t) & B(t)\end{pmatrix},\quad
N_\theta^\sigma=\frac{\theta(t^\sigma)}{t^\sigma}
\begin{pmatrix}0 & 1\\ A(t^\sigma) & B(t^\sigma)\end{pmatrix}.
\]
The first row of $\Lambda$ is given by $\lambda_0(t),\lambda_1(t)$. 
To determine
the second row we apply $\theta$ to $\cartier(1/f)=\lambda_0(t)/f^\sigma+
\lambda_1(t)\theta(1/f)^\sigma$. We get $\cartier(\theta(1/f))=
\mu_0(t)/f^\sigma+\mu_1(t)\theta(1/f)^\sigma$, where $\mu_0(t),\mu_1(t)\in R$
and $\mu_1(t)$ is divisible by $p$.

Let $Y$ be a fundamental matrix solution of the system $\theta Y=N_\theta Y$. Then
$Y^\sigma =Y(t^\sigma)$ is a fundamental matrix solution of $\theta Y^\sigma=N_\theta^\sigma Y^\sigma$.
Write $\Lambda(t)=Y\Lambda_0(Y^\sigma)^{-1}$ for some unknown matrix $\Lambda_0$ and
substitute this in the differential equation in Proposition \ref{Cartier-connection-relation}.
A straightforward calculation shows that $\theta\Lambda_0=0$, hence the entries of $\Lambda_0$
are constants. Let us take 
\[
Y=\begin{pmatrix}F(t) & F(t)\log t+G(t)\\ \theta F(t) & \theta(F(t)\log t+G(t))
\end{pmatrix}=
\begin{pmatrix}F(t) & G(t)\\ (\theta F)(t) & F(t)+(\theta G)(t)\end{pmatrix}
\begin{pmatrix} 1 & \log t\\ 0 & 1\end{pmatrix}
\]
and its $\sigma$-image,
\[
Y^\sigma=\begin{pmatrix}F(t^\sigma) & G(t^\sigma)\\ (\theta F)(t^\sigma) & 
F(t^\sigma)+(\theta G)(t^\sigma)\end{pmatrix}
\begin{pmatrix} 1 & \log t^\sigma\\ 0 & 1\end{pmatrix}.
\]
Since $Y\Lambda_0(Y^\sigma)^{-1}$ has entries in $\Z_p\pow t$, the 
terms with $\log t$ should cancel out. This can only happen if $\Lambda_0$
has the form 
\[
\Lambda_0=\begin{pmatrix}\alpha_0 & \alpha_1\\ 0 & p\alpha_0\end{pmatrix}
\]
for certain $\alpha_0,\alpha_1\in\Q_p$, and in this case one has
\[
\begin{pmatrix} 1 & \log t\\ 0 & 1\end{pmatrix} \begin{pmatrix}\alpha_0 & \alpha_1\\ 0 & p\alpha_0\end{pmatrix} \begin{pmatrix} 1 & - \log t^\sigma\\ 0 & 1\end{pmatrix} = \begin{pmatrix} \alpha_0 & \alpha_1 - \log(t^\sigma/t^p)\\ 0 & 1\end{pmatrix}.
\]
Note that substituting $t=0$ in this matrix we should obtain $\Lambda(0)$. 
Using the values of $\lambda_0(0)$ and $\lambda_1(0)$ found in Proposition~\ref{cartier-on-example} we conclude that $\alpha_0=1$ and $\alpha_1=\log(\gamma^{p-1})$, where $\gamma$ is the vertex coefficient of $g(\v x)$. A straightforward computation
then shows that $\lambda_1(t)$, being the upper right entry in $\Lambda$, is equal to
\be{lambda1-formula}
\lambda_1(t)=\frac{pF(t)F(t^\sigma)}{W(t^\sigma)}
\left(\frac1p \log(\gamma^{p-1}) - \frac{1}{p}\log\left(\frac{t^\sigma}{t^p}\right)+
\frac{G(t)}{F(t)}-\frac{1}{p}\frac{G(t^\sigma)}{F(t^\sigma)}\right),
\ee
where
\[
W(t)=\det(Y)=F^2+F(\theta G)-(\theta F)G
\]
is called the 
{\it Wronskian determinant}. Furthermore, we also find that
\be{lambda0-formula}
\lambda_0(t)=\frac{F(t)}{F(t^\sigma)}-\frac{(\theta F)(t^\sigma)}{F(t^\sigma)}\lambda_1(t).
\ee

\begin{remark}
From \eqref{lambda0-formula} and \eqref{lambda1-formula} it follows that
\[
\cartier\left(\frac{1}{f}\right)\is \frac{F(t)}{F(t^\sigma)}\frac{1}{f^\sigma}+\lambda_1(t)
\left(\theta\left(\frac{1}{f}\right)-\frac{(\theta F)}{F}\frac{1}{f}\right)^\sigma
\mod{p^2\fil_2^\sigma}.
\]
Note that the factor following $\lambda_1(t)$ lies in $\fil_1^\sigma$. Although this will not
be used, we thought it nice enough to observe.
\end{remark}

We collect some useful properties in the following proposition.

\begin{proposition}\label{hw1-hw2}
With the notations as above we have $W(t)\in1+t\Z_p\pow t$ and 
\[
hw^{(2)}(t)\is W(t)^{1-p}\mod{p}.
\] 
\end{proposition}

\begin{proof}
Notice that a priori
$W(t)\in 1+t\Q_p\pow t$. Take the determinant on both sides of 
$\Lambda=Y\Lambda_0(Y^\sigma)^{-1}$. We get $\frac{1}{p}\det(\Lambda)=\frac{W(t)}{W(t^\sigma)}$.
Since $\det(\Lambda)=\lambda_0(t)\mu_1(t)-\lambda_1(t)\mu_0(t)$ is divisible by $p$,
we get that $W(t)/W(t^\sigma)=\phi(t)$ with $\phi(t)\in\Z_p\pow t$ and $\phi(0)=1$.
Hence $W(t)=\phi(t)\phi(t^\sigma)\ldots$,
which is in $1+t\Z_p\pow t$.

From \eqref{HW-action} we know that $\Lambda\is HW^{(2)}(CY(g))\mod{p^2}$. Taking
determinants, $p \, hw^{(2)}(t)\is\det(\Lambda(t))\mod{p^2}$. Divide by $p$ on both sides
and use $\frac{1}{p}\det(\Lambda(t))=W(t)/W(t^\sigma)$ to conclude that
$hw^{(2)}(t)\is W(t)/W(t^\sigma)\is W(t)^{1-p}\mod{p}$. 
\end{proof}

Note that $W(t)^{1-p}$ mod $p$ is a polynomial as a consequence of
Proposition \ref{hw1-hw2}.

To prove Corollary \ref{mirror-p-integral} we recall the following fact.
\begin{lemma}[{\bf Dieudonn\'e-Dwork lemma}]\label{dieudonne-dwork}
Let $g(t)\in t\Q_p\pow t$. Then
\[
g(t)-\frac{1}{p}g(t^\sigma)\in\Z_p\pow t
\]
if and only if $\exp(g(t))\in\Z_p\pow t$.
\end{lemma}
\begin{proof}
Define $a_1,a_2,\ldots\in\Q_p$ by $\exp(g(t))=\prod_{k\ge1}(1+a_kt^k)$. One easily sees by induction
on $k$ that $\exp(g(t))\in\Z_p\pow t$ if and only if $a_k\in\Z_p$ for all $k\ge1$. 

Define $\phi(t)=pg(t)-g(t^\sigma)$ and rewrite $\phi$ in terms of the $a_k$,
\[
\phi(t)=\sum_{k\ge1}\log\left(\frac{(1+a_kt^k)^p}{1+a_k(t^\sigma)^k}\right).
\]
It suffices to prove that $\phi(t)\in p\Z_p\pow t$ if and only if $a_k\in\Z_p$ for all $k\ge1$. 
First note that if $a_k\in\Z_p$, then 
\[
\frac{(1+a_kt^k)^p}{1+a_k(t^\sigma)^k}\is\frac{1+a_kt^{pk}}{1+a_kt^{pk}}\is1\mod{p},
\]
hence
\[
\log\left(\frac{(1+a_kt^k)^p}{1+a_k(t^\sigma)^k}\right)\in p\Z_p\pow t.
\]
In particular $\phi(t)\in p\Z_p\pow t$ if $a_k\in\Z_p$ for all $k\ge1$. 

Suppose conversely that $\phi(t)\in p\Z_p\pow t$. The coefficient of $t$ in $\phi(t)$ equals $pa_1$.
Hence $a_1\in\Z_p$. We now prove by induction on $k$ that $a_k\in\Z_p$. Suppose $k\ge1$ and $a_1,\ldots,a_k\in\Z_p$.
Then
\[
\phi_k(t):=\phi(t)-\sum_{r=1}^k\log\left(\frac{(1+a_rt^r)^p}{1+a_r(t^\sigma)^r}\right)\in p\Z_p\pow t.
\]
The coefficient of $t^{k+1}$ in $\phi_k(t)$ reads $pa_{k+1}$. Hence $a_{k+1}\in\Z_p$, which completes our
induction step. 
\end{proof}

\begin{corollary}\label{mirror-p-integral}
Define the power series expansion $q(t)=t\exp(G(t)/F(t))$. 
Then $q(t)\in\Z_p\pow t$.
\end{corollary}

\begin{proof}
Note that the Cartier upper right entry $\lambda_1(t)$ is given by
formula \eqref{lambda1-formula}. By Proposition \ref{cartier-on-example} 
we know that $\lambda_1(t)\in p\Z_p\pow t$. Together with $F(t),W\in\Z_p\pow t^\times$ 
and $\frac{1}{p}\log(\gamma^{p-1})\in\Z_p$ this implies that
\[
\frac{1}{p}\log(t^p/t^\sigma)+\frac{G(t)}{F(t)}-
\frac{1}{p}\frac{G(t^\sigma)}{F(t^\sigma)}\in\Z_p\pow t.
\]
Since $t^p/t^\sigma \in 1 + p \Z_p\lb t \rb$, this implies
\[
\frac{G(t)}{F(t)}-
\frac{1}{p}\frac{G(t^\sigma)}{F(t^\sigma)}\in\Z_p\pow t.
\]
Application of the Dieudonn\'e-Dwork lemma then gives our conclusion.
\end{proof}

\begin{remark} Recall that $F(t)$ and $F(t) \log t + G(t)$ are solutions of the second order differential operator defined in Proposition~\ref{second-order-de}. 
The power series $q(t)= t \exp(G(t)/F(t))$ is called the \emph{canonical coordinate}.
Its inverse power
series $t(q)$ is known as the \emph{mirror map}. 
The statement of the above corollary is clearly equivalent to $p$-integrality of the mirror map.
In the case of hypergeometric equations of so-called MUM-type, integrality of the mirror map is
proven by Krattenthaler and Rivoal, \cite{KR10}. They use explicit congruences
for hypergeometric coefficients studied by Dwork.
\end{remark}

In the simple example of Section~\ref{sec:example1} we had seen that there
exists a Frobenius lift $\sigma$ which renders the action of $\cartier$
modulo $\fil_2$ particularly simple. It turns out that in the setting of
completely symmetric Calabi-Yau families there is also such an excellent lift.

\begin{lemma}\label{lambda-1-via-q} Let $q\in\Z_p\pow t$ be the canonical coordinate defined above. Then 
\[
\lambda_1(t) = \frac{F(t)F(t^\sigma)}{W(t^\sigma)}\log\left( \frac{\gamma^{p-1} q^p}{q^\sigma} \right).
\]
\end{lemma}
\begin{proof} In the right-hand side we substitute $q(t)= t \exp(G(t)/F(t))$ and 
verify that we recover formula~\eqref{lambda1-formula}.
\end{proof}

\begin{definition}\label{excellent-lift}
We call $\sigma$ such that $q^\sigma=\gamma^{p-1}q^p$ the excellent Frobenius lift 
on $\Z_p\lb t \rb$.
\end{definition}

We note that due to Corollary~\ref{mirror-p-integral} one has $\Z_p\lb t \rb=\Z_p\lb q \rb$, and therefore taking for $q^\sigma$ any element congruent to $q^p$ modulo $p$ defines a Frobenius lift on this ring.

\begin{proof}[Proof of Theorem \ref{excellent-example2}]
Let $\sigma$ be the excellent Frobenius lift on $\Z_p\lb t \rb$ as in Definition~\ref{excellent-lift}. Consider the series $t^\sigma \in \Z_p\lb t \rb$. It is given by $t(q^\sigma)=t(\gamma^{p-1}q^p) \in \Z_p\lb q \rb$ and subsequent substitution of $q=q(t)$. If $t^\sigma \in R$, then the excellent lift restricts to $R$ and Lemma~\ref{lambda-1-via-q} implies that $\lambda_1(t) = 0$. Note that Lemma~\ref{lambda-1-via-q} also shows that, conversely, we can have $\lambda_1(t)=0$ only if the Frobenius lift on $R$ is the restriction of the excellent lift. When $\lambda_1(t)=0$ then formula~\eqref{lambda0-formula} yields $\lambda_0(t)=F(t)/F(t^\sigma)$ and the desired congruence~\eqref{congruence-mod-fil2} follows from Proposition~\ref{cartier-on-example}. 

It therefore remains to prove that $t^\sigma \in R$. Since $hw^{(1)}(t)$ and $hw^{(2)}(t)$ are units in $R$, it follows that $R$ contains the $p$-adic completion of $\Z_p[t,1/hw^{(1)}(t),1/hw^{(2)}(t)]$. In Theorem~\ref{excellent-is-analytic} below we show that $t^\sigma$ belongs to this latter ring, which completes our proof. 
\end{proof}

\begin{theorem}\label{excellent-is-analytic}
Let $R$ be the
$p$-adic completion of the ring $\Z_p[t,1/hw^{(1)}(t),1/hw^{(2)}(t)]$. Let $\sigma$ be the
excellent Frobenius lift on $\Z_p\lb t \rb$ as in Definition~\ref{excellent-lift}. 
Then $t^\sigma \in \Z_p\lb t \rb$ is an 
element of $R$.   
\end{theorem}

Recall that $hw^{(1)}(t)$ and $hw^{(2)}(t) \mod p$, and hence also the ring $R$ in this theorem,
do not depend on the Frobenius lift $\sigma$. At the beginning of this section we noted that
\[
hw^{(1)}(t) \is \sum_{k=0}^{p-1}\text{constant term of }t^kg(\v x)^k\is F_p(t)\mod{p},
\]
where $F_p$ is the $p$-truncation of 
the power series $F(t)$. Another observation, it follows from \eqref{congruence-mod-ps}
with $m=s=1$ that $hw^{(1)}\is F_p(t)\is F(t)^{1-p}\mod{p}$. 
 
Similarly, recall from Proposition \ref{hw1-hw2} that $hw^{(2)}\is W(t)^{1-p}\mod{p}$,
where $W(t)$ is the Wronskian determinant of the differential equation introduced
in Proposition~\ref{second-order-de}. In analogy with the previous we believe that
$hw^{(2)}(t)$ equals the $p$-truncation of $W(t)$ modulo $p$. 

\begin{proof}[Proof of Theorem \ref{excellent-is-analytic}]
Recall that $t^\sigma$ is defined by 
\[
\log(\gamma^{p-1})+\log (t^p/t^\sigma)+p\frac{G(t)}{F(t)}-\frac{G(t^\sigma)}{F(t^\sigma)}=0.
\]
Furthermore, by \eqref{lambda1-formula} applied with the Frobenius lift
$t\to t^p$, we get
\[
\log(\gamma^{p-1})+p\frac{G(t)}{F(t)}-\frac{G(t^p)}{F(t^p)}=\frac{W(t^p)}{F(t^p)F(t)}\lambda_1(t),
\]
where $\lambda_1(t)$ lies in $pR$ and is associated to the Frobenius lift $t\to t^p$.
Subtract the first equation from the second to get
\[
\log (t^\sigma/t^p)+\frac{G(t^\sigma)}{F(t^\sigma)}-\frac{G(t^p)}{F(t^p)}
=\frac{W(t^p)}{F(t^p)F(t)}\lambda_1(t).
\]
Since $t^\sigma=t(\gamma^{p-1}q^p)$ we see that it can be written as $t^\sigma=t^p+t^ph(t)$ 
for some power series $h \in p \Z_p\lb t \rb$. Substituting this expression into the above
left hand side and expanding it in a Taylor series 
around $t^p$ we get
\[
\sum_{m\ge1}\left[ t^{m}\left(\frac{d}{dt}\right)^m\left(\log t+\frac{G(t)}{F(t)}\right)\right] \Big|_{t=t^p}\frac{h(t)^m}{m!}.
\]
Notice that
\be{first-derivative}
\theta\left(\log t+\frac{G(t)}{F(t)}\right)=\frac{W(t)}{F(t)^2}.
\ee
Let us define $H(t)=\frac{W(t)}{F(t)^2}$. We know that $\theta F/F\in R$ and 
$\theta W/W=B(t)\in R$.
Here $B(t)$ is the coefficient of the second order equation introduced in 
Proposition~\ref{second-order-de}.
Hence
\be{second-derivative}
\frac{\theta H}{H}=\frac{\theta W}{W}-2\frac{\theta F}{F}\in R.
\ee
Using \eqref{first-derivative} and \eqref{second-derivative} we can show by induction that 
for every $m\ge1$ there exists $\tilde{r}_m(t)\in R$ such that
\[
t^m\left(\frac{d}{dt}\right)^m\left(\log t+\frac{G(t)}{F(t)}\right)=\tilde{r}_m(t)H(t).
\]
In particular, $\tilde{r}_1=1$. Let us write $r_m=\tilde{r}_m(t^p)$. Clearly $r_m \in R$.
Our implicit equation for $h(t)$ becomes
\[
H(t^p)\sum_{m\ge1}\frac{r_m}{m!}h(t)^m=\frac{W(t^p)}{F(t^p)F(t)}\lambda_1(t),
\]
Hence 
\[
\sum_{m\ge1}\frac{r_m}{m!}h(t)^m=\frac{F(t^p)}{F(t)}\lambda_1(t).
\]
Using Lemma \ref{inverse-series} below we apply the reverse series on both sides
to obtain
\[
h(t)=\sum_{m\ge1}\frac{s_m}{m!}\left(\frac{F(t^p)}{F(t)}\lambda_1(t)\right)^m
\]
with $s_m\in R$ for $m\ge1$. This proves our theorem because $F(t^p)/F(t)\in R$
and $\lambda_1(t)\in pR$.
\end{proof}

\begin{lemma}\label{inverse-series}
Let $R$ be a ring. Consider the series $P(z)=\sum_{m\ge1}r_m\frac{z^m}{m!}$ with $r_m\in R$
for all $m\ge1$ and $r_1=1$. Then the reverse power series has the form $Q(z)=
\sum_{m\ge1}s_m\frac{z^m}{m!}$ with $s_m\in R$ for all $m\ge1$.
\end{lemma}

\begin{proof}
For every $m\ge1$, apply the $m$-th derivative $\frac{d^m}{dz^m}$ to the expression
$Q(P(z))=z$ and substitute $z=0$. This way one obtains an expression for $s_m
=Q^{(m)}(0)$ as a $\Z$-linear combination of products of $s_i$ with $i<m$
and $r_j$ with $j\le m$.
\end{proof}

As we remarked at the beginning of this section, we were unable to prove a mod $p^{2s}$ version
of \eqref{congruence-mod-ps}. However, using coefficients of Laurent expansions of $1/f$ it is
possible to deduce supercongruences that look similar. We give an example here.

\begin{proposition}
Consider the hypercubic family in $n$ variables. Let $p$ be an odd prime. Consider for any odd integer $Q$
the polynomial in $1/t$ defined by
\[
P_Q(t)=t^{-Q}\sum_{k=0}^{(Q-1)/2}\left(\frac{(2k-Q)\cdots(k+1-Q)}{k!}\right)^n t^{2k}.
\]
Let $\sigma$ be the excellent Frobenius lift. Then
\[
P_{p^s}(t)\is\frac{F(t)}{F(t^\sigma)}P_{p^{s-1}}(t^\sigma)\mod{p^{2s}}
\]
for any $s\ge1$.
\end{proposition}

\begin{proof}
Expand $1/f$ as power series in $x_1,\ldots,x_n$. 
A straightforward computation shows that the coefficient of $(x_1\cdots x_n)^Q$
equals $P_Q(t)$. Take $Q=p^s$ and apply Theorem \ref{excellent-example2}. Expand $1/f$ and
$1/f^\sigma$ in \eqref{congruence-mod-fil1} modulo $p^2\fil_2$ and consider the coefficient of
$(x_1\cdots x_n)^{p^s}$ on both sides. Our proposition follows immediately. 
\end{proof}

\section{The excellent Frobenius lift}
In 'p-Adic cycles', \cite{dwork69} Dwork devotes the last sections 7 and 8
to the algebraic relations of degree $p+1$ between modular functions of
the form $h(\tau)$ and $h(p\tau)$. Dwork observed that the Frobenius lift
$h(\tau)\to h(p\tau)$ is a particular lift with very desirable properties.
Eventually, in the context the of the Legendre family of elliptic curves,
he calls them 'excellent lifts'. We have seen examples of such lifts in
Theorems \ref{theorem-example1} and \ref{excellent-example2}. 

In this section we shall
display explicit excellent lifts for hypercubic families in the cases $n=1,2,3,4$. 

Let us start with $n=1$.
Let $f=1-t(x+1/x)$. Explicit computation shows that 
\[
\theta((1-4t^2)\theta-4t^2)\frac{1}{f}=\theta_x^2\frac{1}{f}\in\fil_2.
\]
A basis of solutions of the equation $\theta((1-4t^2)\theta-4t^2)y=0$ is given
by
\[
\frac{1}{\sqrt{1-4t^2}},\quad \frac{1}{\sqrt{1-4t^2}}\times \log\left(\frac{2t}{1+\sqrt{1-4t^2}}\right).
\]
Straightforward computation shows that
\be{q-expansion-nis1}
t(q)=\frac{q}{1+q^2}=\frac{1}{q+1/q}.
\ee
Hence Theorem \ref{excellent-example2} implies
\[
\cartier\left(1-\frac{x+1/x}{q+1/q}\right)^{-1}\is
\frac{1+q^2}{1-q^2}\frac{1-q^{2p}}{1+q^{2p}}
\left(1-\frac{x+1/x}{q^p+1/q^p}\right)^{-1}\mod{\fil_2}.
\]
It turns out we have exact equality.
\begin{proposition}
For all primes $p$ we have
\[
\cartier\left(1-\frac{x+1/x}{q+1/q}\right)^{-1}=
\frac{1+q^2}{1-q^2}\frac{1-q^{2p}}{1+q^{2p}}
\left(1-\frac{x+1/x}{q^p+1/q^p}\right)^{-1}.
\]
\end{proposition}
\begin{proof}
One easily derives that
\[
\left(1-\frac{1}{q+1/q}(x+1/x)\right)^{-1}=\frac{1+q^2}{1-q^2}
\left(\frac{1}{1-qx}-\frac{1}{1-q^{-1}x}\right).
\]
Since $\cartier(1/(1-Ax))=1/(1-A^px)$ for any $A$ we get as Cartier image
\[
\frac{1+q^2}{1-q^2}\left(\frac{1}{1-q^px}-\frac{1}{1-q^{-p}x}\right)
\]
and our result follows. 
\end{proof}

Let us turn to the case $n=2$.
The functions $F(t),F(t)\log t+G(t)$ satisfy the Gaussian 
hypergeometric equation for ${}_2F_1(1/2,1/2,1|z)$, but with $z$ replaced by $t^2$.
The map
$t\mapsto \tau:=\frac{F(t)\log t+G(t)}{F(t)}$ is called the Schwarz map. 
Classical Schwarz theory of hypergeometric functions tells us that the inverse of
the Schwarz map is related to a modular function of level $2$.
More concretely, in the first case we get
\[
t(q)=q-4q^3+14q^5-40q^7+101q^8-\cdots=q\prod_{n\ge1}\frac{(1+q^{4n})^8}{(1+q^{2n})^4},
\]
Observe that $t(q)^2=\lambda(q^2)/16$, where $\lambda$ is the classical Klein modular 
$\lambda$-function. We know that $\lambda(e^{\pi i\tau})$ is a Hauptmodul for the modular
congruence subgroup $\Gamma(2)$. For $t$ itself this means that $t(e^{\pi i\tau/2})$ is a modular
function with respect to $\Gamma(4)$. 
By a classical identity of Felix Klein we have
\[
F(t(q))={}_2F_1(1/2,1/2,1|\lambda(q^2))=\left(\sum_{k\in\Z}q^{2k^2}\right)^2,
\]
which is a weight 1 modular form with respect to $\Gamma(2)$ when we take $q=e^{\pi i\tau/2}$.

For any odd prime $p$ the function $t(q^p)$ with $q=e^{\pi i\tau/2}$ is modular with respect
to the group $\Gamma(4)\cap\Gamma_0(p)$, which has index $p+1$ in $\Gamma(4)$. 
Therefore $t^\sigma=t(q^p)$ is an algebraic function of $t(q)$ of degree $p+1$. 
The polynomial $\Phi_p(X,Y)\in\Z[X,Y]$ such that $\Phi_p(t^\sigma,t)=0$ is called the modular polynomial.
Explicit computation yields for example
\[
\Phi_3(X,Y)=X^4 - X Y + 12 X^3 Y + 6 X^2 Y^2 + 12 X Y^3 - 256 X^3 Y^3 + Y^4
\]
and
\begin{eqnarray*}
\Phi_5(X,Y)&=&X^6 - X Y + 20 X^3 Y - 70 X^5 Y - 40 X^2 Y^2 + 655 X^4 Y^2\\
&& +  20 X Y^3 - 660 X^3 Y^3 + 5120 X^5 Y^3 + 655 X^2 Y^4 -  10240 X^4 Y^4\\
&& - 70 X Y^5 + 5120 X^3 Y^5 - 65536 X^5 Y^5 + Y^6.
\end{eqnarray*}

A straightforward calculation shows that the family of elliptice curves $1-t(x+1/x)(y+1/y)=0$ is
isomorphic over $\Q(i,t)$ to the Legendre family $y^2=x(x-1)(x-16t^2)$. 
In \cite[\S7]{dwork69} it is shown that in this case $(t^\sigma)^2$ is contained in the 
$p$-adic completion of the ring $\Z_p[t,1/hw^{(1)}]$ (so without $hw^{(2)}\is(1-16t^2)^{p-1}\mod{p}$). This
is referred to as 'Deligne's theorem' by Dwork. 

Now we turn to $n=3$. Then $F(t)={}_3F_2(1/2,1/2,1/2;1,1|64t^2)$. By Clausen's formula, we know that
\[
{}_3F_2(1/2,1/2,1/2;1,1|z)={}_2F_1(1/4,1/4;1|z)^2={}_2F_1(1/2,1/2;1|(1-\sqrt{1-z})/2)^2.
\]
Hence there should again be a relation with modular functions through the
properties of Gauss's hypergeometric functions. 
It turns out that
\[
t(q)=q-12q^3+78q^5-376q^7+\cdots=\prod_{n\ge1}\frac{(1+q^{4n})^{12}}
{(1+q^{2n})^{12}}
\]
and
\[
F(t(q))=\left(\sum_{k\in\Z}q^{2k^2}\right)^4.
\]
The latter is clearly a modular form of weight 2.

In cases when $n\ge4$ we do not expect any modular behaviour.

One of the interesting aspects of Theorem \ref{excellent-is-analytic} is that, given
$t_0\in\Z_p$ such that $hw^{(1)}(t_0),hw^{(2)}(t_0)\in\Z_p^\times$, we 
can compute $t^\sigma$ at the point $t_0$ by $p$-adic approximation. Let us call
this value $t^\sigma(t_0)$. This gives us the possibility to specialize Theorem
\ref{excellent-example2} to values $t_0\in\Z_p$ (or in a finite extension).
Although the possibility of such a specialization is present, it is usually not obvious
how to compute $t^\sigma(t_0)$ explicitly. Except in a few special cases which we
briefly discuss below. In that discussion we shall be interested in $t_0$ such that
$t^\sigma(t_0)=t_0$. In such a case Theorem \ref{excellent-example2} specializes to
\begin{corollary}
Let notations be as in Theorem \ref{excellent-example2} and suppose that $t^\sigma(t_0)=t_0$.
Suppose also that $hw^{(1)}(t_0),hw^{(2)}(t_0)\in\Z_p^\times$. 
Let $f_0(\v x)=1-t_0g(\v x)$. Then
\[
\cartier\left(\frac{1}{f_0}\right)\is u_0\frac{1}{f_0}\mod{p^2\fil_2},
\]
where $u_0$ is $F(t)/F(t^\sigma)$ evaluated at $t=t_0$. This is the unit root of
the $\zeta$-function of the curve $f_0(\v x)=0$. 
\end{corollary}
Here are a few instances of explicit solutions of $t^\sigma(t_0)=t_0$ in the case of
hypercubic families. 

When $n=1$ it follows from \eqref{q-expansion-nis1} that $t_0=1/2$ implies $q=1$ and hence $t^\sigma(1/2)=1/2$.
The choice $t_0=1$ implies that $q=e^{\pm2\pi i/3}$. Hence $q^p=e^{\pm2\pi i/3}$, which implies $t^\sigma(1)=1$.
Similarly we obtain $t^\sigma(-1)=-1$.

When $n=2$ we have seen that $t(q)$ is a modular function. We also have the modular relation
$\Phi_p(t(q^p),t(q))=0$. So in particular $\Phi_p(t^\sigma(t_0),t_0)=0$, an algebraic equation
for $t^\sigma(t_0)$.

From classical computation it follows that the value $t_0=1/4$ corresponds to the cusp $q=1$.
Hence $t^\sigma(1/4)=1/4$. Non-cuspidal values $t_0$ of $t(q)$ such that $t^\sigma(t_0)=t_0$
satisfy $\Phi_p(t_0,t_0)=0$. They correspond to 
elliptic curves $f_0(\v x)=0$ which have complex multiplication. 
A particular CM-value is $t_0=i/4$, which is attained by $t(\tau)$ at
$\tau=(1+i)/2$. The curve $1-\frac{i}{4}(x+1/x)(y+1/y)=0$ is an elliptic
curve which is isomorphic over $\Q(i)$ to $Y^2=X^3-X$, a CM-curve with 
endomorphism ring $\Z[i]$. Note that
\[
F_p(i/4)=\sum_{k=0}^{(p-1)/2}\binom{-1/2}{k}^2(-1)^k
\]
is the prototype of the Hasse-Witt invariant which was already studied by Igusa.
Since $Y^2=X^3-X$ has ordinary reduction at $p$ if $p\is1\mod{4}$, we find that 
$F_p(i/4)\in\Z_p^\times$ when $p\is1\mod{4}$. 

Finally, when $n\ge3$ in the hypercubic case, we conjecture that $t^\sigma(1/2^n)=1/2^n$.

\section{Appendix}
In this section we display two linearly independent 
solutions of the differential equation 
\be{second-order-de-displayed}
\theta^2y-B(t)\theta y-A(t)y=0,
\ee 
which was introduced in Proposition~\ref{second-order-de}.
This differential equation describes the connection on the quotient of
the completely symmetric Calabi-Yau crystal $CY(g)$ modulo $\fil_2$.
We have $f(\v x)=1-t g(\v x)$ and the coefficients $A,B \in \Z_p\lb t \rb$
are determined by the condition that $\theta^2(1/f) - B \theta(1/f) - A/f \in \fil_2$.
The result below shows in particular that they do not depend on $p$. 

\begin{proposition}\label{second-order-solutions}
Let assumptions and notations be as in Theorem \ref{excellent-example2}. Denote the vertices of 
$\Delta$ by $\v v_1,\v v_2,\ldots,\v v_N$ and write 
$g(\v x) = \alpha + \gamma \sum_{i=1}^N \v x^\v v_i$ with $\alpha, \gamma \in \Z$.
Let $\Gamma$ be the lattice of vectors 
$\bm\ell=(\ell_1,\ell_2,\ldots,\ell_N)\in\Z^N$ such that $\ell_1\v v_1+\cdots+\ell_N\v v_N=\v0$.
We denote $|\bm\ell|=\sum_{i=1}^N\ell_i$.
Consider the power series in $t$,
\[
F(t)=\sum_{m=0}^\infty\sum_{\bm\ell\in\Gamma\atop\ell_1,\ldots,\ell_N\ge0}
t^{m +|\bm\ell|}\alpha^m \gamma^{|\bm \ell|}\frac{(m+|\bm\ell|)!}{m!(\ell_1)!\cdots(\ell_N)!}
\]
and 
\[\bal
G(t)= &\sum_{m=0}^\infty\sum_{\bm\ell\in\Gamma\atop\ell_1,\ldots,\ell_N\ge0}
t^{m +|\bm\ell|}\alpha^m \gamma^{|\bm \ell|}
\left(\sum_{r=\ell_1+1}^{m+|\bm\ell|}\frac{1}{r}\right)
\frac{(m+|\bm\ell|)!}{m!(\ell_1)!\cdots(\ell_N)!} \\
&+\sum_{m=0}^\infty\sum_{\bm\ell\in\Gamma\atop\ell_1<0,\ell_2,\ldots,\ell_N\ge0}
(-1)^{l_1+1}t^{m +|\bm\ell|}\alpha^m \gamma^{|\bm \ell|}
\frac{(m+|\bm\ell|)!(-1-l_1)!}{m!(\ell_2)!\cdots(\ell_N)!}.
\eal\]
Then $F(t)$ and $F(t)\log t+G(t)$ are linearly independent solutions of equation \eqref{second-order-de-displayed}.
\end{proposition}

We remark that the second term in $G(t)$, with $\ell_1<0$, does not occur
in the simplicial and octahedral examples, but it does occur in the hypercubic case when $n\ge3$.
We start with a lemma which shows that $F$ and $G$ are indeed power series and we have
$F(0)=1$ and $G(0)=0$. 

\begin{lemma}\label{bounded-ell1}
There exists a positive real number $\lambda$
such that $\lambda(\ell_1+\cdots+\ell_N)\ge \ell_2+\cdots+\ell_N$ for all $\bm\ell\in\Gamma$ with
$\ell_2,\ldots,\ell_N\ge0$. 
In particular, for such $\bm\ell \in \Gamma$ we have $|\bm \ell|>0$ if $\bm \ell \ne \v 0$.
\end{lemma}

\begin{proof}
When $\ell_1\ge0$ we can take $\lambda=1$. Suppose $\ell_1<0$. Then $|\ell_1|\v v_1=\ell_2\v v_2+\cdots+\ell_N\v v_N$.
Divide on both sides by $\ell_2+\cdots+\ell_N$. Then we see that $|\ell_1|\v v_1/(\ell_2+\cdots+\ell_N)$ lies
in the convex hull $H$ of $\v v_2,\ldots,\v v_N$. Let $\mu$ be the largest real number such that $\mu\v v_1$
lies in $H$. Since $\v v_1$ is a vertex of $\Delta$, it does not lie in $H$, hence $\mu<1$. So we find
that $-\ell_1=|\ell_1|\le \mu(\ell_2+\cdots+\ell_N)$. Hence $\ell_1+\cdots+\ell_N\ge(1-\mu)(\ell_2+\cdots+\ell_N)$.
Our Lemma follows by taking $\lambda=1/(1-\mu)$. 
\end{proof}

\begin{proof}
We expand $1/f$ in a Laurent series at $\v 0$, namely
\[\bal
\frac{1}{f} &= \sum_{k=0}^\infty t^k g(\v x)^k\\
& =\sum_{m=0}^\infty\sum_{k_1,\ldots,k_N\ge0}t^{m +\sum_i k_i}\alpha^m \gamma^{\sum_i k_i}
\frac{(k_1+\cdots+k_N)!}{(k_1)!\cdots(k_N)!}\v x^{\sum_ik_i\v v_i}.
\eal\]
We determine the coefficient of $\v x^{Q\v v_1}$ for any integer $Q\ge0$.
The set of all $N$-tuples
$(k_1,\ldots,k_N)$ such that $Q\v v_1=\sum_{i=1}^N k_i\v v_i$ is given by 
$(Q+\ell_1,\ell_2,\ldots,\ell_N)$
with $(\ell_1,\ell_2,\ldots,\ell_N)\in\Gamma$.
Since $k_1,k_2,\ldots,k_N\ge0$ in the summation we see that $\ell_2,\ldots,\ell_N\ge0$
and $\ell_1 \ge -Q$. 
The desired coefficient equals $(t\gamma)^Qb(Q,t)$, where
\[
b(Q,t)=\sum_{m=0}^\infty\sum_{\bm\ell\in \Gamma\atop Q+\ell_1,\ell_2,\ldots,\ell_N\ge0}
t^{m +|\bm\ell|}\alpha^m \gamma^{|\bm \ell|}
\frac{(Q+m+|\bm\ell|)!}{m!(Q+\ell_1)!(\ell_2)!\cdots(\ell_N)!}.
\]
This expression can be rewritten as
\[\bal
b(Q,t)&=\sum_{m=0}^\infty\sum_{\bm\ell\in \Gamma\atop\ell_2,\ldots,\ell_N\ge0}
t^{m +|\bm\ell|}\alpha^m \gamma^{|\bm \ell|}\frac{(Q+m+|\bm\ell|)\cdots(Q+\ell_1+1)}
{m!(\ell_2)!\cdots(\ell_N)!}\\
&=\sum_{M=0}^{\infty} B_M(Q) t^M
\eal\]
with some coefficients $B_M(Q) \in \Q[Q]$. The fact that $B_M(Q)$ are polynomials
follows from~Lemma \ref{bounded-ell1} which shows that only finitely many terms
contribute to each coefficient at a given power of $t$. Note that 
$b(0,t)=F(t) \in \Z\lb t \rb$ is the constant term in the above expansion of $1/f$.

Take the coefficient of $\v x^{Q\v v_1}$ in the expansion of 
\be{connection-relation}
\theta^2(1/f) - B(t) \theta(1/f) - A(t)/f \in \fil_2.
\ee
When $Q=0$ this yields
$\theta^2F-B(t)\theta F-A(t)F=0$, so $F$ is indeed a solution
of~\eqref{second-order-de-displayed}. 
We assert that we get $F(t)\log t+G(t)$ by taking the partial derivative
of $t^Qb(Q,t)$ with respect to $Q$
and then setting $Q=0$. It will take a small detour to see that it satisfies our
second order equation.

Let $Q=p^s$. Taking the coefficient of $\v x^{p^s\v v_1}$ in~\eqref{connection-relation} 
yields
\[
\theta^2(t^{p^s}b(p^s,t))-B(t)\theta (t^{p^s}b(p^s,t))-A(t)t^{p^s}b(p^s,t)\is 0\mod{p^{2s}}.
\]
Hence
\be{congruence-de}
(\theta+p^s)^2b(p^s,t)-B(t)(\theta+p^s) b(p^s,t)-A(t)b(p^s,t)\is 0\mod{p^{2s}}.
\ee
Let $N$ be a positive integer and consider \eqref{congruence-de} modulo $t^{N+1}$.
Denote by $b^{(N)}(Q,t)=\sum_{M=0}^N b_M(Q)t^M$ the
truncation of the power series $b(Q,t)$ in $t$ at $t^N$. It is a polynomial in both $Q$ and $t$ 
with coefficients in $\Q$. Let $D_N$ be the common denominator of all 
coefficients of this polynomial.
We still have $b^{(N)}(0,t)\in\Z[t]$. Let $\rho=\ord_p(D_N)$ and suppose that
$s>\rho$. We deduce that 
\[
b^{(N)}(p^s,t)\is b^{(N)}(0,t)+\frac{\partial b^{(N)}}{\partial Q}(0,t)p^s\mod{p^{2s-\rho}}.
\]
Substitution of this expansion in \eqref{congruence-de} modulo $t^{N+1}$ yields
\begin{eqnarray*}
&&[(\theta+p^s)^2-B(t)(\theta+p^s)-A(t)]b^{(N)}(0,t)\\
&&+p^s[(\theta+p^s)^2-B(t)(\theta+p^s)-A(t)]\frac{\partial b^{(N)}}{\partial Q}(0,t)
\is0 \; \mod{p^{2s-\rho},t^{N+1}}.
\end{eqnarray*}
Modulo $p^{2s-\rho}$ the first line equals 
\[
[\theta^2-B(t)\theta-A(t)]b^{(N)}(0,t)+(2p^s\theta-p^sB(t))b^{(N)}(0,t)\mod{p^{2s-\rho}}
\]
and since $b^{(N)}(0,t)\is F(t)\mod{t^{N+1}}$ we are left with
$p^s(2\theta-B(t))F(t)\mod{p^{2s-\rho},t^{N+1}}$. 
Divide by $p^s$ to get
\[
(2\theta-B(t))F(t)+[\theta^2-B(t)\theta-A(t)]
\frac{\partial b^{(N)}}{\partial Q}(0,t)\is0\mod{p^{s-\rho},t^{N+1}}.
\]
Let $s\to\infty$ to see that the above equality holds true modulo $t^{N+1}$ for any integer $N$.
By letting $N\to\infty$ we find that
\[
(2\theta-B(t))F(t)+[\theta^2-B(t)\theta-A(t)]\frac{\partial b}{\partial Q}(0,t)=0.
\]
This can be rewritten as
\[
[\theta^2-B(t)\theta-A(t)]\left(F(t)\log t+\frac{\partial b}{\partial Q}(0,t)\right)=0.
\]
The proposition is proved once we have verified that $G(t)=\frac{\partial b}{\partial Q}(0,t)$.
This is a straightforward computation.  
\end{proof}

\end{document}